\documentclass[11pt, a4paper]{amsart}
\usepackage[utf8]{inputenc}
\usepackage{cite}
\usepackage[top=3cm, bottom=3cm, left=3.5cm, right=3.5cm] {geometry}
\usepackage[T1]{fontenc}
\usepackage{amsmath}
\usepackage{amsthm}
\usepackage{amsfonts}
\usepackage[english]{babel}
\usepackage{array,epsfig}
\usepackage{amssymb}
\usepackage{amsxtra}
\usepackage{amsthm}
\usepackage{latexsym}
\usepackage{dsfont}
\usepackage{mathrsfs}
\usepackage[mathscr]{eucal}
\usepackage{color}
\usepackage[all]{xy}
\usepackage{hyperref}
\usepackage{graphicx}
\usepackage{mathtools}
\usepackage{caption}
\usepackage{relsize}
\usepackage{float}

\usepackage{tikz}
\tikzstyle{color}=[circle,draw=black!50,fill=black!20,thick, inner sep=0pt,minimum size=1mm]

\theoremstyle{plain}
\newtheorem{defn}{Definition}[section]
\newtheorem{thm}[defn]{Theorem}
\newtheorem*{theorem*}{Theorem}

\theoremstyle{definition}
\newtheorem{cor}[defn]{Corollary}
\newtheorem*{rmk}{Remark}

\newtheorem{lem}[defn]{Lemma}
\newtheorem{ex}[defn]{Example}

\newtheorem{prop}[defn]{Proposition}
\newtheorem*{conjecture}{Conjecture}

\newcommand{\set}[1]{\left\{#1\right\}}
\newcommand{\tuple}[1]{\left(#1\right)}

\newcommand{\abs}[1]{\left|#1\right|}

\newcommand{\sprod}[1]{\left<#1\right>}

\newcommand{\ol}[1]{\overline{#1}}

\newcommand{\wt}[1]{\widetilde{#1}}

\renewcommand{\phi}{\varphi}

\newcommand{\del}{\partial}

\newcommand{\Ra}{\Rightarrow}
\newcommand{\Lra}{\Leftrightarrow}

\newcommand{\Cbb}{\mathbb{C}}

\newcommand{\Fbb}{\mathbb{F}}

\newcommand{\Nbb}{\mathbb{N}}

\newcommand{\Pbb}{\mathbb{P}}
\newcommand{\Qbb}{\mathbb{Q}}
\newcommand{\Rbb}{\mathbb{R}}
\newcommand{\Sbb}{\mathbb{S}}

\newcommand{\Zbb}{\mathbb{Z}}

\newcommand{\Ccal}{\mathcal{C}}
\newcommand{\Dcal}{\mathcal{D}}
\newcommand{\Ecal}{\mathcal{E}}
\newcommand{\Fcal}{\mathcal{F}}
\newcommand{\Gcal}{\mathcal{G}}

\newcommand{\Ical}{\mathcal{I}}

\newcommand{\Kcal}{\mathcal{K}}
\newcommand{\Lcal}{\mathcal{L}}
\newcommand{\Mcal}{\mathcal{M}}
\newcommand{\Pcal}{\mathcal{P}}
\newcommand{\Ocal}{\mathcal{O}}

\newcommand{\Vcal}{\mathcal{V}}

\newcommand{\Ncal}{\mathcal{N}}

\newcommand{\gfrak}{\mathfrak{g}}
\newcommand{\hfrak}{\mathfrak{h}}
\newcommand{\kfrak}{\mathfrak{k}}
\newcommand{\lfrak}{\mathfrak{l}}
\newcommand{\pfrak}{\mathfrak{p}}
\newcommand{\mfrak}{\mathfrak{m}}
\newcommand{\qfrak}{\mathfrak{q}}
\newcommand{\tfrak}{\mathfrak{t}}

\newcommand{\wfrak}{\mathfrak{w}}
\newcommand{\zfrak}{\mathfrak{z}}

\newcommand{\rank}{\operatorname{rank}}

\newcommand{\Hom}{\operatorname{Hom}}

\newcommand{\Aut}{\operatorname{Aut}}

\newcommand{\Bl}{\operatorname{Bl}}
\newcommand{\MA}{\operatorname{MA}}
\newcommand{\vol}{\operatorname{vol}}
\renewcommand{\bar}{\operatorname{bar}}

\makeatletter
\@namedef{subjclassname@1991}{\textup{2020} Mathematics Subject Classification}
\makeatother

\title{Spherical cones: classification and a volume minimization principle}
\author{Tran-Trung Nghiem}
\keywords{Spherical cones, Conical Calabi-Yau metrics, Real Monge-Ampère equations, Variational approach, Weighted volume minimization.}
\subjclass{14M27, 32Q25, 32Q26, 35J96}
\address{Tran-Trung Nghiem, IMAG, Univ Montpellier, CNRS, Montpellier, France}
\email{tran-trung.nghiem@umontpellier.fr}

\begin{document}
\begin{abstract}
Using a variational approach, we establish the equivalence between a weighted volume minimization principle and the existence of a conical Calabi-Yau structure on horospherical cones with mild singularities. This allows us to do explicit computations on the examples arising from rank-two symmetric spaces, showing the existence of many irregular horospherical cones. 
\end{abstract}
\maketitle
\section{Introduction}

\subsection{Background, motivation and main results} 

Since the resolution of the Calabi conjecture by Yau and Aubin, the problem of finding canonical metrics has been essential in Kähler geometry. In the last three decades, various authors formulated and showed that the existence of Kähler-Einstein metrics on Fano manifolds is equivalent to a purely algebro-geometric condition called K-stability \cite{Tia97, Don02, CDSa, CDSb, CDSc}. Even if the corresponding picture of non-compact varieties is still largely unexplored, partial results have been made in the case of normal affine varieties.  

A normal \( \Qbb \)-Gorenstein affine variety with klt singularities and a good action of a complex torus \( T \) is often called a \textit{Fano cone} (or a Fano cone singularity as in \cite{LWX}). Here, by a \textit{good action}, we mean that the action of the torus is effective with a unique fixed point that is contained in the closure of any \( T \)-orbit. Basically, a Fano cone is the affine cone over a log Fano variety. The log Fano base is smooth if and only if the cone is smooth outside the fixed point.

Consider a Fano cone \( X \) with a unique singularity at the fixed point \( \set{0_X}\). Let \( \xi \) be a vector in the Lie algebra of the maximal compact torus \( T_c \subset T \) such that \( -J \xi \) is a homothetic-scaling vector field on \( X \). The element \( \xi \) is called a \textit{Reeb vector}. A \textit{conical Calabi-Yau metric} on \( (X, \xi) \) is a Kähler metric \( \omega \) on \( X \backslash \set{0_X} \) which is Ricci-flat and \( 2 \)-homogeneous under \( - J \xi \), i.e.
\[ \Lcal_{-J\xi} \omega = 2 \omega \] 
When \( X \) is toric, a result of Martelli-Sparks-Yau relates the existence of a conical Calabi-Yau metric to a \textit{volume minimizing} condition \cite{MSY08}. This was generalized to any Fano cone with an isolated singularity by Collins and Székelyhidi \cite{CS19}, broadening the volume minimization principle (as \( \xi \) varies) to a suitable notion of K-stability for the pair \( (X, \xi) \). The Fano toric case with non-isolated singularities was settled recently by Berman \cite{Ber20}, recovering and extending the main result of \cite{MSY08}. Finally, a generalized notion of volume minimization was obtained by Li for \textit{any} Fano cone and shown to be in fact equivalent to K-semistability \cite{Li17}.  Interestingly, a conical Calabi-Yau structure on a Fano cone can be translated to a particular \(g\)-soliton structure on quasi-regular quotient orbifolds \cite{HL}.  In algebro-geometric words, the K-stability of the cone is equivalent to the \( g \)-weighted K-stability of any quotient orbifold. Both conditions can be verified over special test configurations only \cite[Theorem 2.9]{Li21}. 

Even then, the K-stability condition still remains hard to be checked in practice, since there is generally an infinite number of special test configurations to consider. However, on varieties with low-complexity group actions (e.g. toric varieties), the condition becomes more or less simplified. For instance, the central fiber of any test configuration of a toric Fano variety is always isomorphic to the variety itself, hence the vanishing of the Futaki invariant (which is equivalent to the barycenter of a polytope being zero) is necessary and sufficient for K-stability. The main objective of this article is to conduct a case study of conical Calabi-Yau metrics on a class of spherical cones (with non-isolated singularities in general), called \textit{horospherical cones}, and find an equivalent combinatorial K-stability notion for these cones. 

To do this, we first provide a classification of conical embeddings using ingredients from the Luna-Vust theory. Before stating the classification theorem, let us give some preliminaries. Let \( G \) be a complex connected linear reductive group.

\begin{defn}
A \emph{spherical space} is a homogeneous space \( G/H \) containing a Zariski-open orbit of a Borel subgroup \( B \). 
A \emph{\( G/H \)-spherical embedding} is a pair \( (X,x) \) where \( X \) is a normal \( G \)-variety such that \( G.x \) is an open \( G \)-orbit in \( X \) and \( H \) is the stabilizer of \( x \). 
\end{defn} 
Such a variety is called \textit{spherical}. A \textit{spherical cone} is a spherical affine variety with a fixed point under the action by automorphisms of a non-trivial torus that commutes with \( G \). If a spherical embedding \(G/H \subset  X \) is moreover a spherical cone, then \( X \) is said to be a \textit{conical embedding} of \( G/H \). Spherical cones form a large class of affine cones which notably contains the toric cones. The description of the latter is very simple: a toric cone can be completely characterized by a polyhedral convex cone of maximal dimension. 

Analogously, the spherical embeddings of \( G/H \) are classified by combinatorial objects in a vector space, called \textit{colored fans}. A colored fan is a collection of strongly convex and polyhedral colored cones satisfying certain compatibility conditions (see Defn. \ref{definition_colored_cones}, Defn. \ref{definition_colored_fans} for the precise definitions). A \textit{color} in \( G/H \) is an irreducible \(B\)-stable divisor in \( G/H \). 

\begin{thm} \label{main_theorem_classification}
Let \( Y \) be a conical embedding of a spherical space \( G/H \).  Let \( \Vcal(G/H) \) be the valuation cone of \( G/H \), i.e. the set of \( G \)-invariant valuations on the rational functions of \( G/H \). Then the following assertions are equivalent: 
\begin{enumerate}
    \item \( Y\) is a spherical cone. 
    \item \( Y \) has a unique fixed point under the action of \( G \).
    \item The valuation cone \( \Vcal(G/H) \) contains a line, and the colored cone of \( Y \) is \( ( \Ccal_Y, \Dcal) \) where \( \Ccal_Y \) is a strictly convex polyhedral cone of maximal dimension, and \( \Dcal \) is the set of all colors of \( G/H \). 
\end{enumerate}
\end{thm}

A horospherical variety is a spherical variety whose open \( G \)-orbit is an equivariant torus bundle over a flag variety. The class of horospherical cones, which has a rather simple description, contains strictly the class of toric cones. Moreover, horospherical cones can be shown to satisfy a Yau-Tian-Donaldson correspondence.
\begin{thm} \label{main_theorem_volume_minimization}
Let \( Y \) be a \( \Qbb \)-Gorenstein \( G \)-horospherical cone and \( T_H \) the connected component of the group of automorphisms commuting with \( G \). Then the following are equivalent: 
\begin{enumerate}
    \item  \( Y \) admits a \( \xi \)-conical Calabi-Yau metric. 
    \item  \( (Y, T_{H}, \xi) \) is K-stable. 
    \item  \( \xi \) minimizes the normalized Duistermaat-Heckman volume of \( Y \). 
\end{enumerate}
\end{thm}

Our proof of Theorem \ref{main_theorem_volume_minimization}, which is based on \cite{Ber20}, uses the variational approach by Berman-Berndtsson \cite{BB13} to solve a weighted real Monge-Ampère equation with exponential right-hand side and a relatively gentle condition on weight. The combinatorial condition is then equivalent to the existence of a locally bounded solution on the open dense orbit \( G/H \). This solution is actually smooth if the weight is smooth (see Thm. \ref{realMA_barycenter}), and extends to a locally bounded conical Calabi-Yau potential over all \( Y\) (thanks to a \( C^0\)-estimate). We then use a regularity result, shown in a companion paper \cite{N22}, to conclude. 

The volume minimization principle implies that there is a unique choice of a Reeb vector giving rise to a conical Calabi-Yau structure on the horospherical cone. This allows us to verify K-stability by explicit computations, thus producing many new examples of irregular Calabi-Yau cones. In the context of string theory, an infinite family of irregular toric cones (with an isolated singularity) was constructed by Gauntlett-Martelli-Sparks-Waldram \cite{GMSW}. These are the first examples of irregular Sasaki-Einstein manifolds, contradicting a conjecture of Cheeger and Tian that such manifolds do not exist. In our article, we construct a possibly infinite family of irregular \( \Qbb\)-Gorenstein horospherical cones, arising more or less unexpectedly from symmetric spaces.  

\begin{thm} \label{main_theorem_irregular_examples}
There exist irregular Calabi-Yau horospherical cones arising from irreducible rank-two symmetric spaces of type \( BC_2 \). 
\end{thm}

In fact, another motivation of our result comes from the problem of finding asymptotic cones for Ricci-flat symmetric spaces. It turns out that we can build horospherical cones from combinatorial data associated to rank-two symmetric spaces. Such conical embeddings, endowed with their conical Calabi-Yau metrics, are expected to be asymptotic cones of suitable Ricci-flat metrics on the symmetric spaces considered in this article. This will be studied in a future paper.

\subsection{Organization} 
\begin{itemize}
\item We start in Section \ref{Spherical varieties} with generalities on the theory of equivariant embeddings of spherical spaces, as well as a brief description of the canonical divisor and \( \Qbb\)-Gorenstein singularities. We then give a combinatorial necessary and sufficient condition for a spherical embedding to be a cone, and show that a \( \Qbb\)-Gorenstein spherical cone is a Fano cone. Several examples of homogeneous spaces (not) admitting conical embeddings are gathered at the end of the section. 
\item Section \ref{CY}  aims to study conical Calabi-Yau metrics on \(\Qbb\)-Gorenstein horospherical cones. We derive an explicit expression of the canonical volume form. The conical Calabi-Yau problem is then translated into a weighted real Monge-Ampère equation with exponential right-hand side, which can be solved using the variational approach of \cite{BB13}.
\item In Section \ref{Examples}, we provide examples of regular and irregular cones arising from irreducible symmetric spaces of rank \( 2 \).
\end{itemize}
\subsection{Notations}
\begin{itemize}
\item In this text, \( G \)  is a complex simply connected linear reductive group, while \( B \), \( T \), \( U \) denote respectively the Borel subgroup, maximal torus of \( G \) and maximal unipotent subgroup of \( B \). Under an embedding of \( G \) into \( GL_n \), \( B, T, U\) can be identified with subgroups of upper-triangular, diagonal, and upper-triangular matrices with ones on the diagonal, respectively.  Likewise, the opposite Borel subgroup \( B^{-} \) (resp. opposite maximal unipotent subgroup \( U^{-} \)) can be identified with the subgroup of lower-triangular matrices (resp. lower-triangular matrices with ones on the diagonal). 

\item \( G^{0}\) denotes the connected component of a group \( G \). 

\item A \( (G,B,T) \)-root system is a root system \( (G,T)\) whose positive roots are the roots of \( (B,T) \). 

\item Letters in fraktur are used to denote the corresponding Lie algebras. The reductivity of \( G \) is equivalent to the existence of a real maximal compact subgroup \( K \) such that \( \gfrak = \kfrak \oplus i \kfrak \). 

\item Throughout the article, by a variety we mean an \textit{irreducible complex algebraic variety} (or an integral separated scheme of finite type over \( \Cbb\)). A smooth point on the variety is understood to be smooth in the algebraic sense (hence smooth in the analytic sense). 

\item Recall that a variety is complete (in the algebraic sense) if and only if it is compact for the Euclidean topology. To avoid confusion with the metric notion of completeness, a complete algebraic variety will be called a \textit{compact variety} in this paper. 

\item Given a set \( E \) in a real vector space, the \textit{cone generated by E}, denoted by \( \Rbb_{\geq 0} E \), is the set of all finite nonnegative linear combinations of elements in \( E \). 

\item Given a lattice \( \Mcal \) and its dual lattice \( \Ncal = \Hom(\Mcal,\Zbb) \), we denote by \( \sprod{.,.} \) the natural pairing between them. 
\end{itemize}

\textbf{Acknowledgements.} This paper is part of a thesis prepared under the supervision of Thibaut Delcroix and Marc Herzlich, partially supported by ANR-21-CE40-0011 JCJC project MARGE. I would like to thank Thibaut Delcroix for many illuminating discussions and remarks. 

\section{Spherical varieties} \label{Spherical varieties}

Our main references in this section are \cite{Bri97}, \cite{Kno91}. Recall that a spherical space is a homogeneous space \( G/H \) containing a Zariski-open orbit of a Borel subgroup \( B \subset G\). The only new results in this section are Thm.  \ref{spherical_cone_criterion} and Prop. \ref{spherical_cone_g_fixed_point}. 

\subsection{Classification of embeddings} 

\subsubsection{Combinatorial data of a spherical space}

Let \( G/H \) be a spherical space. By a theorem of Chevalley, we may view \( G/H \) as a smooth quasiprojective variety. Let \( \Cbb(G/H)^{(B)} \) be the set of rational functions on \( G/H \) which are eigenvectors of \( B \). An element of \( \Cbb(G/H)^{(B)} \) is called a \( B\)-semi-invariant function. The set of \( B \)-invariant rational functions \( \Cbb(G/H)^B \) is actually \( \Cbb \) by Rosenlicht's theorem. 

\begin{defn}
To a spherical space are associated two lattices. 
\begin{itemize}
    \item The weights \( \Mcal \) of \( \Cbb(G/H)^{(B)}\), called the \emph{weight lattice}.  
    \item The dual lattice \( \Ncal := \Hom(\Mcal,\Zbb) \), called the \emph{coweight lattice}.  
\end{itemize}
 We will use \( \Mcal_{\Rbb} \) and \( \Ncal_{\Rbb} \) to denote the real vector spaces \( \Mcal \otimes \Rbb\) and \( \Ncal \otimes \Rbb \). The dimension of \( \Mcal_{\Rbb} \) is called the \emph{rank of \( G/H \)}.  
\end{defn}

Remark that there exists a natural bijection between \( \Mcal  \) and \( \Cbb(G/H)^{(B)}/\Cbb^{*} \), which sends a weight to its eigenvector. Indeed, let \( f_1, f_2 \) be two functions of the same weight, then \( f_1/ f_2 \) is a \( B \)-invariant rational function, hence constant. 

The open \( B \)-orbit in \( G/H \), which is isomorphic (as an affine variety) to \( (\Cbb^{*})^k \times \Cbb^m \) \cite[Theorem 5]{Ros63}, is an open affine subset of \( G/H \), hence its complement in \( G/H \) is a collection of \( B \)-stable irreducible divisors. They are the only \( B \)-stable irreducible divisors of \( G/H \). 

\begin{defn}
The set $\Dcal$ of irreducible $B$-stable divisors of $G/H$ is called the \emph{set of colors of $G/H$}.
\end{defn}

After perhaps conjugating \( H \), we can identify the open \( B \)-orbit with \( BH/H \).

Recall that a \textit{rational valuation} of a normal variety \( X \) is a map \( \nu : \Cbb(X) \to \Qbb \) satisfying: 
\begin{itemize}
    \item \( \nu (f_1 + f_2 ) \geq \nu(f_1) + \nu(f_2) \), for all \( f_1, f_2 \in \Cbb(X) \backslash \set{0} \) such that \( f_1 + f_2 \neq 0 \). 
    \item \( \nu(f_1 f_2 ) = \nu(f_1) + \nu(f_2) \) for all \( f_1, f_2 \in \Cbb(X) \backslash \set{0} \). 
    \item \( \nu(f) = 0 \) if and only if \( f \) is constant. 
\end{itemize}
If \( X \) is a \( G\)-variety, the valuation is said to be \textit{\(G\)-invariant} if \( \nu(g.f) = \nu(f) \) for all \( g \in G \) and \( f \in \Cbb(X) \), where \( (g.f)(.) := f(g^{-1}.) \). 

\begin{prop} \cite{Kno91}
Denote by \( \Vcal \) the set of \( G\)-invariant rational valuations of \( G/H\). 
\begin{itemize}
    \item The natural mapping
\begin{align*}
\rho : \Vcal &\to \Hom(\Mcal, \Qbb) = \Ncal_{\Qbb} \\
\nu & \to (\chi \to \nu(f_{\chi}) ) 
\end{align*}
is a well-defined injection. In particular, we can identify \( \Vcal \) with a cone in \( \Ncal_{\Rbb} \), called the \textit{valuation cone} of \( G/H \).  
    \item The valuation cone \( \Vcal \) is a strictly convex and polyhedral cone in \( \Ncal_{\Rbb} \).
\end{itemize}
\end{prop}

\begin{defn}
Each element in \( \Dcal \) induces a valuation of \( \Cbb(G/H) \), hence a natural map \( \Dcal \to \Ncal_{\Qbb} \) as above (but non-injective in general), still denoted by \( \rho \). We call \( \rho(\Dcal) \) the \emph{images of the colors of} \( G/H \). 
\end{defn}


\begin{defn}
Let \( G/H \) be a spherical space. We call  \((\Mcal, \Vcal, \Dcal)\) the \emph{combinatorial data} associated to \( G/H \).  
\end{defn}

\begin{ex} \label{rank_one_horospherical}
Let \( G = SL_2 \) act on \( \Cbb^2 \) on the left. The stabilizer \( H\) of \( (1,0) \) is then \(U \). Let \( B \subset SL_2 \) be the Borel subgroup and \( \alpha \) the unique positive root corresponding to \( B \) with coroot \( \alpha^{\vee} \). Let \( B^{-} \) be the opposite Borel subgroup of \( B \).  The open \( G \)-orbit is isomorphic to \( SL_2 / U \simeq \Cbb^2 \backslash \set{0} \), with open dense \( B^{-} \)-orbit isomorphic to \( B^{-} / U \). The fundamental weight of \( B^{-} \) is 
\[ \omega(b_{ij}) = b_{11} \] 
The \( H \)-right-invariant eigenvector of \( B^{-} \) with weight \( \omega \) is then \( f(A \in SL_2) = a_{22} \). It follows that \( \Mcal = \Zbb \omega \). The coweight lattice \( \Ncal \) is \( \Zbb (\alpha^{\vee}|_{\Mcal}) \), and \( \Ncal_{\Rbb} \) coincides with the valuation cone \( \Vcal \). The unique color of \( G/H \) is \( D = \set{f = 0} \). Since \( \rho(D)(f) = \nu_D(f) = 1 \), we have that \( \rho(D) = \alpha^{\vee}|_{\Mcal} \). 
\end{ex}
 
\begin{ex} \label{rank_two_horospherical}
Now let \( G = SL_3 \) and \( H = U^{-} \). The space \( G/H \) is a fibration over the flag \( G/B^{-} \) with fiber the torus \( B^{-} / U^{-} \simeq (\Cbb^{*})^2 \). The open \( B \)-orbit of \( SL_3 / U^{-} \) is isomorphic to \( B / U^{-} \).  Let \( \set{\alpha_1,\alpha_2} \) be the positive simple roots corresponding to \( (G,B,T)\). The fundamental weights \( \omega_1, \omega_2 \) of \( B \) are  
\[ \omega_1(b_{ij}) = b_{11}, \; \omega_2(b_{ij}) = b_{11} b_{22} \] 
Let \( f_1, f_2 \in \Cbb(G/H)^{(B)} \) be two eigenvectors of \( B \) with weights \( \omega_{1,2} \), defined by  \( H \)-invariant functions on \( G \):
\[  f_1(A \in SL_3) = a_{22} a_{33} - a_{23} a_{32}, \; f_2(A) = a_{33}  \]
The \( B \)-stable sets \( D_i = \set{f_i = 0} \) are the colors of \( SL_3 / U^{-} \). 
The lattice \( \Mcal \) is then identified with \( \Zbb \omega_1 \oplus \Zbb \omega_2 \) and its dual \( \Ncal \) is the lattice \( \Zbb \alpha_1^{\vee} \oplus \Zbb \alpha_2^{\vee} \) generated by coroots. In particular, \( \rank(SL_3 / U^{-}) = 2 \). Moreover, \( \rho(D_i)(f_j) = \nu_{D_i}(f_j) = \delta_{ij} \). It follows that \( \rho(D_i) = \alpha_i^{\vee}|_{\Mcal} \). 
\end{ex}

\subsubsection{The Luna-Vust theory of spherical embeddings}

Recall that a spherical embedding is a \( G \)-equivariant embedding of a spherical space \( G/H \). It is clear that a spherical embedding is a spherical variety. Conversely, let \( X \) be a spherical variety with an open orbit \( G.x \). Then we can always choose a stabilizer \( H \) of \( x \) such that \( BH/H \) is open in \( G/H \), hence \( (X,x) \) is a \( G/H \)-spherical embedding. When there is no confusion, we remove \( x \) and simply say that \( G/H \subset X \) is a spherical embedding.

\begin{defn} \label{g_stable_divisors_colors_definition}
Let \(G/H \subset (X,x) \) be a spherical embedding.
\begin{itemize}
\item The divisors of $\Dcal$ whose closure in $X$ contains a closed orbit are called the \emph{colors of $X$}. The set of colors of \( X \) is denoted by \( \Dcal_X \).  
\item We set \( \Vcal_X \) to be the \emph{\( G \)-stable divisors of \( X \)}, which are the irreducible components of \( X \backslash Gx \). If \( D \) is a \( G \)-invariant divisor, then by normality of \( X \), we can associate to it a unique primitive \( G \)-invariant valuation \( \nu_D \), hence under the map \( \rho \), \emph{\( \Vcal_X \) injects to a finite subset of \( \Vcal \)}. 
\end{itemize}
\end{defn}

\begin{rmk}
We distinguish between \textit{the colors of $G/H$} and \textit{the colors of $G/H$ as a trivial embedding of itself}. The latter is always empty.  
\end{rmk} 

\begin{defn}
A spherical variety is said to be \emph{simple} if it contains a unique closed \( G \)-orbit. 
\end{defn}

\begin{prop} \cite{BLV} \label{bstable_affine_chart}
\begin{itemize}
\item Let $X$ be a simple spherical variety  with the unique closed $G$-orbit $Y$. There exists a unique \( B \)-stable affine subset $X_{Y,B} \subset X$, which meets \( Y \) along a Zariski-open \( B\)-orbit in \( Y \). Moreover, \( X = G X_{Y,B} \). 

\item $X \backslash X_{Y,B} = \Dcal \backslash \Dcal_X$ is a union of \( B\)-stable Cartier divisors, generated by global sections, and do not intersect \( Y \). 

\item Every \( G \)-spherical variety admits a finite cover by simple spherical varieties. 
\end{itemize}
\end{prop}

\begin{ex}
Consider Ex. \ref{rank_one_horospherical}. The natural embedding $X = \Cbb^2 \supset G/H$ is an affine simple embedding with the unique closed \( G \)-orbit \(Y = \set{0} \).  The unique open affine \( B^{-} \)-stable subset \( X_{Y,B^{-}} \) containing \(\set{0}\) is \( X \) itself.  
\end{ex}

If we identify the elements in \( \Dcal \) with their closure in \( X \), then the elements of $\Dcal \backslash \Dcal_X$ correspond to $B$-stable divisors of $X$ not containing any \(G\)-orbit. 


\begin{defn} \label{definition_colored_cones} Let \( G/H \) be a spherical space. 
A \emph{colored cone} is a couple $(\Ccal, \Fcal)$ where $\Ccal \subset \Ncal_{\Rbb} $ and  $\Fcal \subset \Dcal$  that verifies: 
\begin{itemize}
    \item $0 \notin \rho(\Fcal)$ and $\Ccal$ is a strictly convex polyhedral cone generated by $\rho(\Fcal)$ (which might be empty) and a finite number of elements of $\Vcal$. 
    \item The relative interior of $\Ccal$ meets $\Vcal$. 
\end{itemize}
A \textit{face} of the colored cone \( (\Ccal, \Fcal) \) is a couple \( ( \Ccal', \Fcal') \) such that \( \Ccal' \) is a face of the cone \( \Ccal \), and that \( \Ccal' \)  meets \( \Vcal \) , and  \( \Fcal' = \Fcal \cap \rho^{-1} ( \Ccal' ) \). 
\end{defn}

The following theorem classifies all the simple embeddings in terms of colored cones. 

\begin{thm}\cite{Kno91}
Let \( G/H \) be a spherical space and \( G/H \subset X \) a spherical embedding with \( \Vcal_X, \Dcal_X \) as in Defn. \ref{g_stable_divisors_colors_definition}.  Let \( \Ccal_X \) be the cone generated by \( \Vcal_X, \rho(\Dcal_X) \). 

The map \( X \to (\Ccal_X, \Dcal_X) \) is a one-to-one correspondence between isomorphism classes of \( G \)-equivariant simple embeddings of \( G/H \) and colored cones. 
\end{thm}

\begin{ex} \label{rank_one_horospherical_simple_embeddings}
Consider the action of \( SL_2 \) in Ex. \ref{rank_one_horospherical}. Apart from the trivial embedding, all possible colored cones are
\begin{itemize}
    \item $\set{\Rbb_{\geq 0} \rho(D), \varnothing}$, which is $X = \Bl_0(\Cbb^2)$.
    \item $\set{\Rbb_{\leq 0} \rho(D), \varnothing}$. The corresponding embedding is $X \simeq \Pbb^2 \backslash [1:0:0]$.
    \item $\set{\Rbb_{\geq 0} \rho(D), \rho(D)}$, with $X  = \Cbb^2$ as the embedding. 
\end{itemize}
Remark that \( \Cbb^2 \) has no \( SL_2 \)-stable divisor and \( \Bl_0(\Cbb^2) \) has no color. 
\end{ex}

\begin{defn} \label{definition_colored_fans}
A \emph{colored fan} is a finite collection $\Fbb$ of colored cones such that: 
\begin{itemize}
\item Every face of a colored cone in \( \Fbb \) belongs to $\Fbb$. 
\item For all $v \in \Vcal$, there exists at most one $(\Ccal, \Fcal) \in \Fbb$ such that $v \in \text{RelInt}(\Ccal)$.
\end{itemize}
The \emph{support} of  \( \Fbb \) is defined as:
\[ \text{Supp}(\Fbb) = \Vcal \cap \cup_{(\Ccal, \Fcal) \in \Fbb} \text{Supp} (\Ccal) \] 
\end{defn}

\begin{thm} \cite{Kno91}  \label{spherical_embeddings_classification}
Given a spherical space $G/H$, there exists a one-to-one correspondence between the isomorphism classes of \( G \)-equivariant embeddings $X \supset G/H$ and the set of colored fans. 

Moreover, $X$ is compact if and only if the support of the fan $\Fbb(X)$ coincides with the valuation cone \( \Vcal \). 
\end{thm}


Let \( X \) be a spherical variety. Any  $G$-orbit 
$Y$  is contained in a unique simple spherical variety  $X_{Y}$, which is a union of the orbits whose closure contains $Y$. By the classification of simple embeddings, \( X_Y \) corresponds to a colored cone \( (\Ccal_{X_Y}, \Dcal_{X_Y}) \). 

\begin{prop}\cite{Kno91}
Let \( Y \) be any \( G \)-orbit in a spherical variety \( X \). There exists a bijection between the $G$-orbits in $X$ whose closure contains the $G$-orbit $Y$  and the faces of the colored cone \( (\Ccal_{X_{Y}}, \Dcal_{X_{Y}}) \), given by $Z \to (\Ccal_{X_{Z}}, \Dcal_{X_{Z}})$. 
\end{prop}

Let \( D \in \Dcal \) be a color of \( G/H \). Its stabilizer \( P_{D} \supset B \) is a parabolic subgroup of \( G \). For all subset \( \Fcal \subset \Dcal \), let \( P_{\Fcal} := \cap_{D \in \Fcal^c} P_D \). In particular, \( P_{\Dcal} = G \). The following is a quite simple dimension formula for a closed orbit of a spherical embedding. 

\begin{prop} \label{dimension_formula} \cite[Theorem 7.6]{Kno91}
Let \( X \) be a simple spherical embedding with a unique closed \( G\)-orbit \( Z \). Let \( (\Ccal_X, \Dcal_X) \) be the colored cone of \( X \). Then 
\[ \dim Z = \text{rank}(X) - \dim \Ccal_X + \dim G / P_{\Dcal_X} \]
If \( X \) is non-simple then we have the same equality for every colored cone in the colored fan of \( X \). 
\end{prop} 

\subsubsection{Some classes of spherical varieties}

\begin{defn}
Let \( X \) be a spherical variety. If $\Dcal_X = \varnothing$, we say that $X$ is a \emph{toroidal variety}. 
\end{defn}

\begin{prop} \label{prop_spherical_dominated_by_toroidal} \cite[3.3]{Bri91}
Every \( G \)-spherical variety \( X \) with colored fan \( \Fcal_X \) is dominated by a \( G\)-toroidal variety \( \wt{X} \) with colored fan \( \wt{\Fcal} = \cup_{\Ccal_X \in \Fcal_X} (\Ccal_X, \varnothing) \), i.e. there exists a proper birational \( G \)-equivariant map \( \pi: \wt{X} \to X \). The datum \( (\wt{X}, \pi) \) is called the \textit{decoloration} of \( X \). 
\end{prop}

\begin{ex}
In Ex. \ref{rank_one_horospherical_simple_embeddings}, the toroidal embedding \(\text{Bl}_{0}(\Cbb^2) \) is the decoloration of \( \Cbb^2 \). 
\end{ex}

\begin{defn}A spherical space \( G/H \) is  called \emph{horospherical} if \( H \) contains a maximal unipotent subgroup of \( G \). 
\end{defn}

\begin{prop} \cite{BP87}  \cite{Pas08} \label{horospherical_space_characterization_proposition}
Let \( (G,B,T) \) be a root system. Let \( \Mcal(B) = \Mcal(T) \) be the weight lattice of \( B \). 
\begin{itemize}
\item A horospherical space \( G/H \) is uniquely determined by a couple \( (\Mcal_I, P_I) \) where \( P_I \) is a parabolic subgroup of \( G \) associated to a subset \( I \) of simple roots of \( G \), and \( \Mcal_I \) is the sublattice of \( \Mcal(B) \) vanishing on coroots of \( I \). 

\item A spherical space $G/H$ is horospherical if and only if the valuation cone $\Vcal$ is the whole vector space $\Ncal_{\Rbb}$. 

\item Moreover, let \( Q \) be the left-stabilizer of the open Borel-orbit and \( P \) its opposite parabolic subgroup. Then \( P = P_I = N_G(H) \) and \( G/H \) is a equivariant torus bundle over the flag variety \( G/P \) with fiber the torus \( P/H \).  
\end{itemize}
\end{prop}

\begin{rmk}
Let \( U \) be a maximal unipotent subgroup of \( B \). If \( H \supset U \), then the horospherical space \( G/H \) has an open dense orbit under \( B^{-} \), isomorphic to \( B^{-}/ B^{-} \cap H\). 
\end{rmk}

\begin{ex}
The spaces \( SL_2/ U \) and \( SL_3 / U^{-} \) in Examples \ref{rank_one_horospherical}, \ref{rank_two_horospherical} are all horospherical. 
\end{ex}

\subsubsection{Affine embeddings} 

\begin{prop} \cite[Theorem 7.7]{Kno91} \label{affine_embedding_criterion}
\begin{enumerate}
\item  A spherical space \( G/H \) is quasi-affine if and only if \( \rho(\Dcal) \) does not contain \( 0 \) and generates a strongly convex cone. 
\item Let \( X \) be a spherical embedding of \( G/H \). Then \( X \) is affine if and only if \( X \) is simple and that there exists \( \chi \in \Mcal \) satisfying 
\[ \chi|_{\Vcal} \leq 0, \quad \chi|_{\Ccal_X} = 0, \quad \chi |_{\rho(\Dcal_X^c)} > 0 \]
In particular, \( G/H \) is affine if and only if \( \Vcal \) and \( \rho(\Dcal) \) are separated by a hyperplane. 
\end{enumerate}
\end{prop}

\begin{ex}
Following Ex. \ref{rank_one_horospherical_simple_embeddings}, the only affine embeddings of the spherical space are the trivial embedding \( \Cbb^{2}\backslash \set{0} \) and \( \Cbb^2 \), as one can readily check from the previous criterion. 
\end{ex}

\subsection{Automorphisms group and canonical divisor}

\subsubsection{Automorphisms group}
Let \( G/H \) be a spherical homogeneous space. We denote by \( N_G(H) \) the normalizer of \( H \) in \( G \), which acts on \( H \) \textit{on the right} by \( p.gH = gp^{-1} H \). The quotient \( N_G(H)/H \) is then identified with a subgroup of \( G \)-equivariant automorphisms of \( G/H \) as follows. 

Let \( \sigma \in \Aut_G(G/H) \) and \( G/H \subset (X,x) \) be a spherical embedding, then \( H\) is also the stabilizer of \( \sigma(x) \). However, \( \sigma \) is only a bijection between the colors of \( (G/H,x) \) and the colors of \( (G/H,\sigma(x)) \) (see \cite{Los09}), but if \( \sigma \in \Aut^0_G(G/H) \), then it is exactly the identity mapping between these sets of colors, therefore \( (X,x) \) and \( (X, \sigma(x)) \) are determined by the same colored fan. It follows that \( \sigma \) extends to a \( G \)-equivariant automorphism of \( X \) by the classification theorem \ref{spherical_embeddings_classification}.  Moreover, by \cite[Proposition 1.8]{Tim11}, we have \( \Aut_G(G/H) \simeq N_G(H)/H \). In particular, 

\begin{prop}  \label{automorphism_group_description} 
Let \( X \) be a spherical embedding of \( G/H \). Then 
\[ \Aut_G^0(X) \simeq \Aut_G^0(G/H) \simeq (N_G(H)/H)^0 \] 
\end{prop}

\begin{prop} \cite[5.2]{BP87} \label{automorphism_group_dimension}
  Let \( \text{lin} \Vcal \) be the linear part of the valuation cone \( \Vcal \). The group \( N_G(H)/H \) is diagonalizable, of dimension equal to \( \dim \text{lin} \Vcal \).
\end{prop}

\subsubsection{Canonical divisor}

\begin{defn}
A normal variety \( X \) is said to be Gorenstein (resp. \( \Qbb\)-Gorenstein) if the canonical divisor \( \Kcal_X \) is Cartier (resp. \( \Qbb\)-Cartier, i.e. \( m \Kcal_X \) is Cartier for some integer \( m > 0 \)). 
\end{defn}

Now let \( G/H \subset X \) be a spherical embedding. Recall that \( \Vcal_X \) and \( \Dcal \) denote the set of \( G \)-stable divisors of \( X \) and colors of \( G/H \), respectively. The following description of the canonical divisor in terms of those data is due to Brion. 

\begin{prop} \cite{Bri89} \label{canonical_divisor_description_prop}
The canonical divisor of \( X \) can be represented by 
\[ \Kcal_X = -\sum_{D_{\nu} \in \Vcal_X} D_{\nu} - \sum_{D \in \Dcal} a_{D} \ol{D} \] 
where \( \ol{D} \) denotes the closure of the color in \( X \) and \( a_D \) depends only on the spherical space \( G/H \).  
\end{prop}

The property of being Cartier for \( \Kcal_X \) can be determined by a linear function. 

\begin{prop} \cite{Bri89} \label{q_gorenstein_condition}
Let \( G/H \subset X \) be a simple embedding with colored cone \( (\Ccal_X, \Dcal_X) \). Then \( X \) is Gorenstein (resp. \( \Qbb\)-Gorenstein) if and only if there exists \( \beta \in \Mcal \) (resp. \(\in \Mcal_{\Qbb}\)) such that for every \( D_{\nu} \in \Vcal_X \) and every \( D \in \Dcal_X \), 
\[ \sprod{\beta, \rho(D_{\nu}) } = 1, \quad \sprod{\beta, \rho(D)} = a_{D}. \] 
\end{prop}

\subsection{Classification of spherical cones} \label{Classification}
If \( G/H \) admits a conical embedding, then it is necessarily quasi-affine. In what follows, we shall consider only quasi-affine spherical spaces, unless mentioned otherwise. 

\subsubsection{Classification and properties}
\begin{defn}
A \emph{spherical cone} is an affine spherical \( G \)-variety \( Y \) that contains a (necessarily unique) fixed point under the right-action of a (non-trivial) torus in \( \Aut_G^0(Y) \simeq (N_G(H)/H)^0 \). 

The fixed point, denoted by \( \set{0_Y} \), is called the \emph{vertex} of \( Y \). If \( Y \) is the embedding of a spherical space \( G/H \), then \( Y \) is said to be a \emph{conical embedding} of \( G/H \). 
\end{defn}

\begin{prop}\cite{Sum74} 
For any spherical cone \( Y \), there exists a \( G \)-equivariant embedding of \( Y \) in an affine space \( \Cbb^N \) such that \( \set{0_Y}\) is sent to \( \set{0}\), \( G \) (resp. \( K \)) corresponds to a subgroup of \( GL_N(\Cbb) \) (resp. \( U_N \)) with their standard action on \( \Cbb^N \), and the torus corresponds to the standard right-action of a diagonal group. 
\end{prop}

\textbf{Notation}
In what follows, we will denote the neutral component of \( N_G(H)/H \) by \( T_H \).

\begin{thm} \label{spherical_cone_criterion}
Let \( Y \) be a spherical embedding of \( G/H \) with colored cone \( (\Ccal_Y,\Dcal_Y) \). Then the following conditions are equivalent: 
\begin{itemize}
\item[i)] \( Y \) is a spherical cone. 
\item[ii)] The valuation cone \( \Vcal \) contains a line, the cone \( \Ccal_Y \) is of maximal dimension and \( \Dcal_Y = \Dcal \). 
\end{itemize}
\end{thm}

\begin{rmk}
Since \( \dim \text{lin} \Vcal \geq 1 \), the rank of a conical embedding is always positive. 
\end{rmk}

\begin{proof}
\( i) \Rightarrow ii) \).  

\textit{a) The valuation cone contains a line.}

If \( Y \) is a cone, then it admits a unique fixed point \( \set{x} \) under the action of \( N_G(H)/H \), so by \ref{automorphism_group_dimension} , \( \dim N_G(H)/H = \dim \text{lin} \Vcal(G/H) \geq 1 \). It follows that \( \Vcal \) contains a line. 

\textit{b) \( \Ccal_Y \) is of maximal dimension.}

The unique fixed point \( \set{x} \) is also a fixed point of \( G \). Indeed, for all \( t \in N_G(H)/H \) and \( g \in G \), 
\[ t.g.x = g.t.x = g.x, \]
hence  \( g.x = x \) by uniqueness of the fixed point. The closed orbit \( \set{x} \) of \( G \) is of dimension \( 0 \), hence by \ref{dimension_formula}, 
\[ 0 = \rank (Y) - \dim \Ccal_Y + \dim G/ P_{\Dcal_Y} \]
This combined with the dimensional equality \( \dim \Ccal_Y + \dim \Ccal_Y^{\perp} = \rank(Y) \) implies that \( \dim \Ccal_Y^{\perp} = \dim G/P_{\Dcal_Y} = 0 \), hence \( \dim \Ccal_Y = \rank(Y) \).    

\textit{c) The colors of \( Y \) are \(\Dcal_Y = \Dcal\) }

Since \( Y \) is affine, by \ref{affine_embedding_criterion}, there exists \( \chi \in \Ccal_Y^{\perp} = \set{0} \) such that \( \chi|_{\Vcal(G/H)} \geq 0 \) and \( \chi|_{\rho(\Dcal_Y)^c} < 0 \). These last conditions are simultaneously possible only if \( \rho(\Dcal_Y^c) = \varnothing \), or equivalently \( \Dcal = \Dcal_Y \).

\( ii) \Rightarrow i) \). 

By assumption, we have \( \Ccal_Y^{\perp} = \set{0} \) and \( \Dcal = \Dcal_Y \), hence the affineness criterion \ref{affine_embedding_criterion} is automatically satisfied for \( Y \). It follows that \( Y \) is a simple embedding, hence contains a unique closed orbit \( Z \). Since \( P_{\Dcal_Y} = P_{\Dcal} = G \), by the dimension formula: 
\begin{align*}
\dim Z &= \rank(Y) - \dim \Ccal_Y + \dim G/P_{\Dcal_Y} \\
&= \rank(Y) - \dim \Ccal_Y = 0 
\end{align*} 	
It follows that \( Z = \set{x} \) is the unique fixed point of \( Y \) under the action of \( G \). Again by assumption, the linear part of \( \Vcal \) is non-empty, hence \( N_G(H)/H \) contains a torus which commutes with \( G \), so fixes \( \set{x} \).  
\end{proof}

\subsection{Proof of Theorem \ref{main_theorem_classification}}
In the proof of our classification theorem \ref{spherical_cone_criterion}, we have shown that a spherical cone has a unique fixed point under the action of \( G \). In fact, the converse is also true. 

\begin{prop} \label{spherical_cone_g_fixed_point}
Let \( Y \) be an affine spherical \( G\)-variety. Then \( Y\) is a cone if and only if \( Y \neq \set{0_Y}\) and contains a fixed point under the action of \( G\). 
\end{prop}

Before giving the proof, we will need some preliminaries. Let \( X \) be a normal \( G \)-variety and \( \nu \) a valuation on \( \Cbb(X)^{*} \). The \textit{center} of \( \nu \) is a closed subvariety \( Z \) of \( X \) such that \( \Ocal_{X,Z} \subset \Ocal_{\nu} = \set{\nu \geq 0} \) and \( \mfrak_{X,Z} \subset \mfrak_{\nu}= \set{\nu > 0} \). Here \( \Ocal_{X,Z} = \set{f/g, f,g \in \Cbb[X], g|_Z \neq 0} \) is a local ring and \( \mfrak_{X,Z} \) is the unique maximal ideal of \( \Ocal_{X,Z} \).  

\begin{defn}
Let \( B \) be a Borel subgroup of \( G \). A \( G \)-invariant discrete  \( \Qbb \)-valuation of \( \Cbb(X) \) is said to be \textit{central} if it vanishes on \( \Cbb(X)^B \backslash \set{0} \). A proper source of \( X \) is a \( G \)-stable proper subvariety of \( X \) which is the center of a central valuation.
\end{defn} 

Given a normal \( G \)-variety \( X \), every closed \( G \)-stable subvariety \( Z \) is the center of some \( G \)-invariant valuation of \( \Cbb(X) \). Indeed let \( (\wt{X}, \pi) \) be the blow-up of \( X \) along \( Z \), then \( \wt{X} \) is still a normal variety. Now let \( \wt{Z} \) be an irreducible component of \( \pi^{-1}(Z) \). By normality of \( \wt{X} \), there exists a natural valuation \( \nu_{\wt{Z}} \) of \( \Cbb(\wt{X}) = \Cbb(X) \) associated to \( \wt{Z} \). The center of \( \nu_{\wt{Z}} \) in \( X \) is then \( Z \).

\begin{prop} \cite[Theorem 8.2]{Kno93}, \cite[Corollary 7.9]{Kno94} \label{prop_proper_source_torus_action}
Let \( X \) be a normal affine \( G \)-variety containing a proper source. Then there exists a non-trivial positive grading of \( \Cbb[X] \) induced by the action of a subtorus \( \Cbb^{*} \subset \Aut_G(X) \), where \( \Aut_G(X) \) is the group of automorphisms of \( X \) commuting with \( G \). 
\end{prop}

\begin{cor}
If \( X \) is a \( G \)-spherical affine variety, containing a proper and closed \( G \)-stable subvariety then  \( \Aut_G(X) \) contains a non-trivial subtorus. 
\end{cor}

\begin{proof}
Indeed, in the spherical case \( \Cbb(X)^B = \Cbb \), hence every \( G \)-invariant valuation of \( \Cbb(X) \) is central. It follows that every proper closed \( G \)-stable subvariety of \( X \) is a proper source. We conclude by the Prop. \ref{prop_proper_source_torus_action}. 
\end{proof}

We now proceed to the proof of Prop.  \ref{spherical_cone_g_fixed_point}

\begin{proof}[Proof of Prop. \ref{spherical_cone_g_fixed_point}]
If \( Y \) is a sherical cone then \( Y \) contains a fixed point under the action of \( G \) by the proof of the part \( ``i) \Rightarrow ii), (b)" \) in Thm. \ref{spherical_cone_criterion}. Conversely, the fixed point of \( G \) in \( Y \) is a proper subvariety of \( Y \), hence by the previous corollary, \( \Aut_G(Y) \) contains a non-trivial torus. It follows that \(\dim \text{lin} \Vcal =  \dim N_G(H)/H  = \dim \Aut_G(Y) \geq 1 \), hence \( \Vcal \) contains a line.  
\end{proof}

\begin{proof}[Proof of Theorem \ref{main_theorem_classification}] Direct consequence of Prop. \ref{spherical_embeddings_classification} and Prop. \ref{spherical_cone_g_fixed_point}.
\end{proof}

\subsection{Relation with Fano cones}

Recall that a \( T \)-affine cone is a normal affine variety with a good action of \( T \) (i.e. \( T \) acts with a fixed point, contained in the closure of any orbit). A Fano cone is a \( \Qbb\)-Gorenstein \( T \)-affine cone with klt singularities (see \cite{Pas17} for a definition and a survey on singularities of spherical varieties).  

\begin{prop}
Let \( Y \) be a spherical cone. Then \( Y \) is a normal \( T_H \)-affine cone. If moreover \( Y \) is \( \Qbb \)-Gorenstein, then \( Y \) is a Fano cone. 
\end{prop}

\begin{proof}
It is clear by definition that \( T_H \) acts effectively and holomorphically on \( Y \) with a unique fixed point \(\set{0_Y}\). 

Next, let us show that \( T_H \) defines a good action. Let \( \Cbb[Y] \) be the ring of regular functions on \( Y \) and \( \Mcal(T_H) \) be the weight lattice of \( T_H \), with dual lattice being \( \Ncal(T_H) \). Under the \textit{right-action} of \( T_H \), we have a decomposition 
\[ \Cbb[Y] = \bigoplus_{\alpha \in \Gamma}  R_{\alpha}, \quad \Gamma := \set{\alpha \in \Mcal(T_H), R_{\alpha} \neq 0} \] 
where \( R_{\alpha} \) is a \( T_H \)-module of \( \Cbb[Y] \) with weight \( \alpha \). The cone \( \sigma^{\vee} \subset \Mcal(T_H)_{\Rbb} \) generated by \( \Gamma \) is strictly convex by Sjammar's criterion \cite[1.2, Corollaire]{Bri97}, and of maximal dimension since \( Y \) has a fixed point under \( T_H \) (see \cite{AH06}). 

It follows from the orbit-faces correspondence for \( T_H \)-affine varieties that \(\set{0_Y}\) correponds to \( \sigma^{\vee}\), and that the closure of any orbit of \( T_H \) in \( Y \) contains the fixed point \cite[1.2, Remarque]{Bri97}. 

Finally, since any \( \Qbb \)-Gorenstein spherical cone has klt singularites (cf. e.g.  \cite[Proposition 5.6]{Pas17} for a short proof), this completes our proof.  
\end{proof}

By duality, the cone \( \sigma = (\sigma^{\vee})^{\vee} \) is also strictly convex, of maximal dimension, and is said to be the \textit{Reeb cone}. In particular, the interior of \( \sigma \) is non-empty and coincides with its relative interior:
\[ \text{Int}(\sigma) = \set{ \xi \in \Ncal(T_H)_{\Rbb}, \sprod{\xi, \alpha} > 0, \forall \alpha \in \Gamma} \] 

\begin{defn}
The set \( \text{Int}(\sigma)\) is called the set of \emph{Reeb vectors} of \( Y \). A couple \( (Y,\xi) \), \( \xi \in \text{Int}(\sigma) \), is said to be a \emph{polarized cone}. 

If \( \xi \in \Ncal(T_H)_{\Qbb} \), then it is called \emph{quasi-regular}, otherwise it is said to be \emph{irregular}. A polarized cone with a quasi-regular (resp. irregular) Reeb vector is called a \emph{quasi-regular} (resp. \emph{irregular}) \emph{cone}. 
\end{defn} 

\begin{rmk}
If \( Y \) is any Fano cone with the Reeb cone defined by the action of \( T \), and \( \xi \) is quasi-regular, then by a result of Kollar \cite[Paragraph 42]{Kol04}, the GIT quotient of \( Y \backslash \set{0_Y} \) by the \( \Cbb^{*} \)-action generated by \( \xi \) can be seen as a polarized variety \( (X,L) \) (more precisely a \textit{log Fano variety}, i.e. \( L = -(K_X + D) \), \( D\) being a Cartier divisor on \( X \) such that \( L \) is ample). The affine cone \( Y \) is then isomorphic to the total space of \( L^{-1} \) with its zero section contracted.
\end{rmk}



\begin{ex} \label{rank_two_horospherical_cone}
Now consider \( G = SL_3 \) and \( H = U^{-} \) as in Ex. \ref{rank_two_horospherical}. Since \( H \) is maximal unipotent, it is a horospherical subgroup of \( G \). 

The embedding defined by the colored cone 
\[\tuple{\Rbb_{\geq 0}(\alpha_1^{\vee}|_{\Mcal}, \alpha_2^{\vee}|_{\Mcal}), \tuple{\alpha_1^{\vee}|_{\Mcal}, \alpha_2^{\vee}|_{\Mcal} }} \]
clearly satisfies the criterion \ref{spherical_cone_criterion}, hence is a conical embedding. The automorphism group \( T_H \) is exactly the maximal torus \( T \) acting on the right.

 We can identify this embedding with the \( 5 \)-dimensional affine quadric 
\[ \mathcal{Q} = \set{ (x_1, x_2, x_3, y_1, y_2, y_3) \in \Cbb^6, \; \sum_{i=1}^3 x_i y_i = 0} \] 
The \( SL_3 \)-orbits of \( \mathcal{Q} \), which correspond to the faces of the colored cone, are the open dense orbit \( SL_3/U^{-} \), two copies of \( \Cbb^3 \), and the unique fixed point \( \set{0} \). In fact, \( \mathcal{Q} \) is the affine cone over the complex grassmannian \( G(2,4) \). In particular, \( \mathcal{Q} \) is a cone with a unique isolated singularity. 

We also have that \( \mathcal{Q} \) is Gorenstein, since the element \( \beta = 2 \omega_1 + 2 \omega_2 \) clearly satisfies \( \sprod{\beta, \alpha_i^{\vee}|_{\Mcal}} = 2, i \in \set{1,2} \). 
\end{ex}
\begin{figure}[H]
\begin{tikzpicture}
\pgfmathsetmacro\ax{2}
\pgfmathsetmacro\ay{0}
\pgfmathsetmacro\bx{2 * cos(90)}
\pgfmathsetmacro\by{2 * sin(90)}
\pgfmathsetmacro\lax{2*\ax/3 + \bx/3}
\pgfmathsetmacro\lay{2*\ay/3 + \by/3}
\pgfmathsetmacro\lbx{\ax/3 + 2*\bx/3}
\pgfmathsetmacro\lby{\ay/3 + 2*\by/3}

\tikzstyle{couleur_pl}=[circle,draw=black!50,fill=blue!20,thick, inner sep = 0pt, minimum size = 2mm]

\fill [black!20] (0,0)--(\ax,\ay)--(\bx,\by) -- cycle;
\node at (\ax, \ay) [couleur_pl] {};
\node at (\bx, \by) [couleur_pl] {};
\draw[->, ultra thick] (0,0) -- (\ax,\ay) node[below right] {\( \rho(D_1) \)};
\draw[->, ultra thick] (0,0) -- (\bx, \by) node[above left] {\(\rho(D_2) \)};

\draw (-2,0)--(2,0);
\draw (0,-2)--(0,2);



\end{tikzpicture}
\caption{Colored cone of a horospherical conical embedding.}
\end{figure}
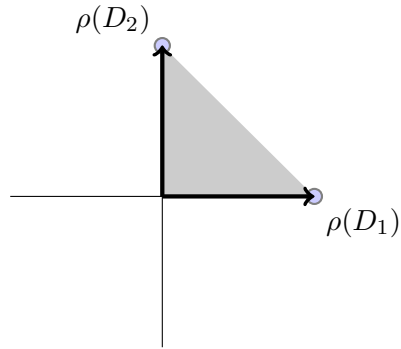

Horospherical spaces of rank \( 0 \) are flag varieties, hence never embed into spherical cones. Below is an example of a rank-one horospherical space not having any conical embeddings. This behavior is completely different from toric spaces, which always admit conical embeddings. 

\begin{ex}
Again we consider \( G = SL_3 \). Fix a root system with simple roots \((\alpha_1,\alpha_2) \) and fundamental weights \( (\omega_1,\omega_2) \) (see Ex. \ref{rank_two_horospherical} for their expressions). Let \( H = \ker(\omega_1 - \omega_2) \). In particular, 
\[ H = \set{(b_{ij}) \in SL_3, b_{22} = 1} \] 
Clearly, \( H \supset U \), so it is a horospherical subgroup. 
The weight and coweight lattices are \( \Mcal = \Zbb \omega_1 \) and \( \Ncal = \Zbb \alpha_1^{\vee} \). 
The images of colors are easily determined to be \( -\alpha_1^{\vee}|_{\Mcal}, \alpha_1^{\vee}|_{\Mcal} \). From these data, one can see that the only embedding with all colors is determined by the colored fan 
\[\set{ \tuple{\Rbb_{\geq 0} (\alpha_1^{\vee}), \alpha_1^{\vee}}, \tuple{\Rbb_{\geq 0} (-\alpha_1^{\vee}), -\alpha_1^{\vee}} }, \] 
This is a non-simple embedding, hence non-conical.
\end{ex}

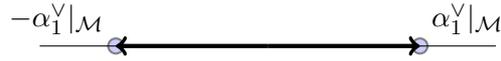
\begin{figure}[H]
\begin{tikzpicture}
\pgfmathsetmacro\ax{2}
\pgfmathsetmacro\ay{0}
\pgfmathsetmacro\bx{2 * cos(90)}
\pgfmathsetmacro\by{2 * sin(90)}
\pgfmathsetmacro\lax{2*\ax/3 + \bx/3}
\pgfmathsetmacro\lay{2*\ay/3 + \by/3}
\pgfmathsetmacro\lbx{\ax/3 + 2*\bx/3}
\pgfmathsetmacro\lby{\ay/3 + 2*\by/3}

\tikzstyle{couleur_pl}=[circle,draw=black!50,fill=blue!20,thick, inner sep = 0pt, minimum size = 2mm]


\node at (\ax, \ay) [couleur_pl] {};
\node at (-\ax, \ay) [couleur_pl] {};
\draw[->, ultra thick] (0,0) -- (\ax,\ay) node[above right] {\( \alpha_1^{\vee}|_{\Mcal} \)};
\draw[->, ultra thick] (0,0) -- (-\ax, \ay) node[above left] {\(- \alpha_1^{\vee}|_{\Mcal} \)};

\draw (-3,0)--(3,0);


\end{tikzpicture}
\caption{Horospherical space without conical embeddings.}
\end{figure}

Not all open orbits of conical embeddings are horospherical, as the following example shows. 

\begin{ex}
Let \( G = SL_2 \times \Cbb^{*} \) and \( H = N \times \set{1} \) where \( N \) is the normalizer of the maximal torus in \( G \). The homogeneous space \( G/H \) is then a reducible symmetric space. Fix a root system of \( SL_2 \) with \( \alpha \) as the unique simple root. Let \( e \) be a primitive character of \( \Cbb^{*} \). The weight lattice is \( \Mcal = \Zbb \gamma \oplus \Zbb e \), where \( \gamma = 2 \alpha \). Let \( (\gamma^{*} = \omega^{\vee}/2, e^{*}) \) be the dual basis, where \( \omega^{\vee} \) is the unique coweight. Then \( \Ncal = \Zbb \gamma^{*} \oplus \Zbb e^{*} \), and there exists a unique color \( D \) whose image in \( \Ncal \) is \( \rho(D) = 2 \gamma^{*} \). 

The valuation cone of \( G/H \) is a half-space defined by the hyperplane 
\[ \set{(x,y) \in \Mcal_{\Rbb} , x \leq 0}, \] 
Consider the embedding defined by \( \tuple{\Rbb_{\geq 0} \set{ 2 \gamma^{*}, e^{*} - 2 \gamma^{*}}, 2 \gamma^{*}}  \). This colored cone is clearly of maximal dimension and contains the unique color, so defines a conical embedding of \( G/H \), which is symmetric and non-horospherical (since the valuation cone is not a vector space). 
\end{ex} 

\begin{figure}[H]
\begin{tikzpicture}
\pgfmathsetmacro\ax{2}
\pgfmathsetmacro\ay{0}
\pgfmathsetmacro\bx{2 * cos(90)}
\pgfmathsetmacro\by{2 * sin(90)}
\pgfmathsetmacro\lax{2*\ax/3 + \bx/3}
\pgfmathsetmacro\lay{2*\ay/3 + \by/3}
\pgfmathsetmacro\lbx{\ax/3 + 2*\bx/3}
\pgfmathsetmacro\lby{\ay/3 + 2*\by/3}

\tikzstyle{couleur_pl}=[circle,draw=black!50,fill=blue!20,thick, inner sep = 0pt, minimum size = 2mm]

\fill [blue!20] (0, -\by) -- (-2, -\by) -- (-2,\by) -- (0, \by) -- cycle;
\node[] at (-1,-0.5) {\( \Vcal \) }; 

\fill [black!20] (0,0)--(\ax,\ay)--(\bx - \ax,\by - \ay) -- cycle;
\node at (\ax, \ay) [couleur_pl] {};
\draw[->, ultra thick] (0,0) -- (\ax,\ay) node[below right] {\( \rho(D) = 2 \gamma^{*} \)};
\draw[->, ultra thick] (0,0) -- (\bx - \ax, \by - \ay) node[above right] {\( . \)};

\draw[dashed] (0,0)--(0,-2);
\node[] at (1,-1) {\(\text{Reeb cone} \)};

\end{tikzpicture}
\caption{Conical embedding of a symmetric space. The dashed half-line represents the Reeb cone}
\end{figure}
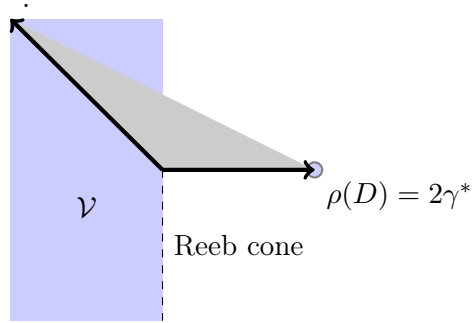

While many familiar spherical spaces admit conical embeddings, there exists a large class which does not, as shown by the following example. 

\begin{ex}
A spherical space \( G/H \) is said to be \textit{sober} if \( H \) is of finite index in \( N_G(H) \). This is equivalent to the fact that \( \Vcal(G/H) \) is a strictly convex cone (cf. Prop. \ref{automorphism_group_dimension}), hence does not contain any line. It follows that sober spaces never admit any conical embedding. In particular, semisimple symmetric spaces do not embed equivariantly into any symmetric cone, since their valuation cone is the restricted negative Weyl chamber, which is strictly convex. 


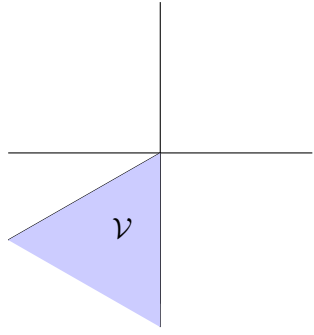
\begin{figure}[H]
\begin{tikzpicture}
\pgfmathsetmacro\ax{2}
\pgfmathsetmacro\ay{0}
\pgfmathsetmacro\bx{2 * cos(120)}
\pgfmathsetmacro\by{2 * sin(120)}
\pgfmathsetmacro\lax{2*\ax/3 + \bx/3}
\pgfmathsetmacro\lay{2*\ay/3 + \by/3}
\pgfmathsetmacro\lbx{\ax/3 + 2*\bx/3}
\pgfmathsetmacro\lby{\ay/3 + 2*\by/3}

\tikzstyle{couleur_pl}=[circle,draw=black!50,fill=blue!20,thick, inner sep = 0pt, minimum size = 2mm]


\draw (-2,0)--(2,0);
\draw (0,-2)--(0,2);

\draw (0,0) -- (-2*\lbx, -2*\lby);
\draw (0,0) -- (-2*\lax, -2*\lay);
\fill [blue!20] (0,0) -- (-2*\lax, -2*\lay) -- (-2*\lbx, -2*\lby) -- cycle;
\node[] at (-0.5,-1) {\( \Vcal \) }; 


\end{tikzpicture}
\caption{Valuation cone of a semisimple symmetric space.}
\end{figure}
\end{ex}

\section{Conical Calabi-Yau metrics on horospherical cones} \label{CY}

\subsection{Preliminaries}

\subsubsection{Structure of horospherical cones}

Let us recall some basic properties of horospherical varieties. For more information, the reader might consult Pasquier's thesis \cite{Pas08}.

\begin{defn}
A horospherical cone is a conical embedding of a horospherical space. 
\end{defn}

Consider a conical embedding \( G/H \subset Y \) of a horospherical space  of rank \( r > 0 \) with colored cone \( (\Ccal_Y, \Dcal_Y) \). 
Let \( K \) be the maximal compact subgroup of \( G\) and \( T \) be a maximal torus of \( G \) such that \( T \cap K \) is the maximal compact subtorus of \( T \).

Denote by \( (.,.) \) the Weyl-invariant scalar product on \( \tfrak \), whose restriction to the semisimple part \( \tfrak \cap [\gfrak,\gfrak] \) coincides with the Killing form. For all \( m_1, m_2 \in \tfrak^{*} \), we denote by \( t_{m_2} \) the unique element in \( \tfrak \) such that 
\[(m_1,m_2) = \sprod{m_1, t_{m_2}} \]
Recall that the connected automorphism group compatible with \( G \) is isomorphic to the torus \( T_H := (N_G(H)/H)^0 \). In the horospherical case, \( N_G(H) = P \) (see Prop. \ref{horospherical_space_characterization_proposition}), so \( T_H = P/H \). In particular, 
\( P \) is the right-stabilizer of the open Borel-orbit. After choosing an adapted Lévi subgroup of \( Q \), (cf. \cite[2.4, Proposition 2]{Bri97}), we have \( Q/H = T/ T \cap H = T_H \). We can thus identify \( T_H \) with \( T/ T \cap H \) as a group, and the weight lattice \( \Mcal(G/H) \) with the weight lattice \( \Mcal(T_H) \) of the torus \( T_H \).  Moreover,  \( r = \text{rank} (G/H) = \dim T_H \).

\begin{rmk} We emphasize here the fact that \( T_H \) acts on \( G/H \) in two different ways: on the left and on the right. In the rest of this article, unless stated otherwise, by an action of \( T_H \), we mean its action on the right, i.e. 
\[ p. (gH) = gp^{-1} H, \; \forall g \in G, p \in T_H \]
By a \( G \times T_H \) action, we mean the combined action of \( G \) on the left and \( T_H \) on the right.
\end{rmk}

Let \( \tfrak_H \) be the Lie algebra of \( T_H \). The compact and non-compact part of \( \tfrak_H \) are denoted by \( \tfrak_{H,c} \) and \( \tfrak_{H,nc} \). They are respectively the Lie algebras of the compact real torus \( T_{H,c} \simeq (\Sbb^1)^r \) and \( T_{H,nc} \simeq  (\Rbb_{>0})^r \) in the decomposition \( T_H = T_{H,c} \times T_{H,nc} \). Moreover, we can choose the maximal compact subgroup \( K \) such that \( T_{H,c} = T_H \cap K \). 

Recall that there exists a natural isomorphism between 
\( \tfrak_{H,c} \) and \( \Ncal(T_H)_{\Rbb} \). 
We identify an element \( \xi \in \tfrak_{H,c} \) in the Reeb cone with the vector field it generates (called the Reeb vector field). Let \( J \) be the complex structure on \( Y_{\text{reg}} \). The vector field generated by \( - J \xi \in \tfrak_{H, nc} \) defines a radial right-action of \( \Rbb^{*}_{+} \) on \( Y \).

\begin{rmk}
The reader should be aware that for a horospherical cone, the Reeb cone is exactly the opposite of the cone \( \Ccal_Y \). Indeed, if we look at the decomposition of \( \Cbb[Y] \) into \( T_H \)-modules with respect to the left-action of \( T_H \), then the cone generated by the weights is exactly \( \Ccal_Y^{\vee} \). However, with respect to the right-action, it is \( - \Ccal^{\vee}_Y \). Below is the Reeb cone for the conical embedding in Ex. \ref{rank_two_horospherical_cone}. 
\end{rmk}

\begin{figure}[H]
\begin{tikzpicture}
\pgfmathsetmacro\ax{2}
\pgfmathsetmacro\ay{0}
\pgfmathsetmacro\bx{2 * cos(90)}
\pgfmathsetmacro\by{2 * sin(90)}
\pgfmathsetmacro\lax{2*\ax/3 + \bx/3}
\pgfmathsetmacro\lay{2*\ay/3 + \by/3}
\pgfmathsetmacro\lbx{\ax/3 + 2*\bx/3}
\pgfmathsetmacro\lby{\ay/3 + 2*\by/3}

\tikzstyle{couleur_pl}=[circle,draw=black!50,fill=blue!20,thick, inner sep = 0pt, minimum size = 2mm]

\fill [black!20] (0,0)--(1.5*\ax,1.5*\ay)--(1.5*\bx,1.5*\by) -- cycle;
\node at (\ax, \ay) [couleur_pl] {};
\node at (\bx, \by) [couleur_pl] {};
\draw[->, ultra thick] (0,0) -- (\ax,\ay) node[below right] {\( \rho(D_1) \)};
\draw[->, ultra thick] (0,0) -- (\bx, \by) node[above left] {\(\rho(D_2) \)};

\draw (-3,0)--(3,0);
\draw (0,-3)--(0,3);

\fill [blue!20] (0,0)--(-1.5*\ax,-1.5*\ay)--(-1.5*\bx,-1.5*\by) -- cycle;
\node[] at (-0.9,-0.5) { \( \text{Reeb cone} \)};


\end{tikzpicture}
\caption{The Reeb and colored cones of a \( SL_3/U^{-} \)-conical embedding.}
\end{figure}
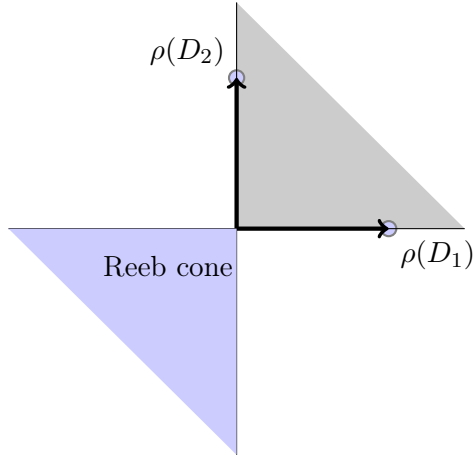

 Denote by \( T_{\xi}  \) the compact torus generated by \( \xi \), which is the closure of \( \exp(\Rbb \xi) \) in \( \Aut_G (Y_{\text{reg}})^{0} \). In particular, \(T_{\xi} \) is a compact subtorus of \(T_{H,c}\). If \( \dim T_{\xi} = 1 \), or equivalently \( \xi \in \Ncal(T_H)_{\Qbb} \), we say that \( \xi \) is \textit{quasi-regular}. If \( \dim T_{\xi} > 1 \) it is called \textit{irregular}.

Let \( \Phi \) be the root system corresponding to $(G,T)$. Let us denote by $\Phi^{+}$ the set of positive roots (with respect to a choice of the Borel subgroup) and $S$ the positive simple roots. Let \( P^u \) (resp. \(Q^u\)) be the nilpotent radical of \( P \) (resp. \( Q \)). The choice of the parabolic group $P = N_G(H)$ is equivalent to the choice of a subset $I \subset S$ (cf. Prop. \ref{horospherical_space_characterization_proposition}). Moreover, let $\Phi_I$ be the root system generated by $I$, the set of $P^u$-roots is then \( \Phi_{P^u} := \Phi^{+} \backslash \Phi_I \). In particular, \( \Phi_{P^u} = - \Phi_{Q^u} \). 
The Lie algebra of \( G \) can be decomposed as:
\[\gfrak = \qfrak^u \oplus \lfrak \oplus \pfrak^u \]  
where \(\qfrak^u\), \( \pfrak^u \) and \( \lfrak \) are the Lie algebras of \( Q^u \), \( P^u \), and the adapted Lévi subgroup \( L = P \cap Q \).
Since $G/H$ is a bundle with fiber \( P/H \simeq T_H \) over \(G/P \), we have:
\[n = r + \abs{\Phi_{P^u}} \]
We will denote by $S_{P^u} :=  \Phi_{P^u} \cap S = S \backslash I$ the set of simple roots of $P^u$ and (two times) the sum of roots in $Q^u$ (resp. \( P^u \) ) by:
$$\varpi_Q = 2 \sum_{\alpha \in \Phi_{Q^u}} \alpha = - \varpi_{P} $$ 
The following proposition describes the colors in a horospherical space in terms of coroots. 
\begin{prop} \cite{Pas08} \label{horospherical_colors_description_prop}
The images of the colors \( \Dcal \) in a horospherical space are exactly the restriction to \( \Mcal \) of the coroots of \( S_{P^u} \):
\[ \rho(\Dcal) = \set{\alpha^{\vee}|_{\Mcal}, \alpha \in S_{P^u} } \]
\end{prop}

Since we only work with quasi-affine spherical spaces, \( \rho(\Dcal) \) does not contain \( 0 \), i.e. for all \( \alpha \in \Phi_{P^u} \), \( \alpha^{\vee}|_{\Mcal} \neq 0 \).

The following proposition follows straightforwardly from our criterion in Thm. \ref{spherical_cone_criterion}. 

\begin{prop}
Let \( Y \) be a horospherical affine variety with open \( G \)-orbit \( G/H \) and colored cone \( (\Ccal_Y, \Dcal_Y) \). The following assertions are equivalent : 
\begin{enumerate}
\item  \( Y \) is a cone. 
\item  \( \Ccal_Y \) is full dimensional and \( \Dcal_Y = \Dcal \). 
\item \( G \) acts on \( Y \) with a fixed point. 
\item \( T_H \) acts on \( Y \) with a fixed point.  
\end{enumerate}
\end{prop}




\subsubsection{Conical Calabi-Yau metrics on horospherical cones}

Let \( Y \) be a \( \Qbb\)-Gorenstein horospherical cone of rank \( r\) and of dimension \( n \) , which is \( G \times T_H \)-equivariantly embedded in \( \Cbb^N \). Let us first state a useful remark. 

\begin{rmk}
By the Iwasawa decomposition of \( G \), cf. \cite[Prop. 2.7]{Del20}, \( \exp(\tfrak_{H, nc})H = T_H \) defines a fundamental domain for the \( K \)-action on \( G/H \), i.e. \( G/H \) can be written as a disjoint union of \( K \)-translated copies of \( T_H \). It follows that every \( K \)-invariant function \( u \) on \( G/H \) restricts on \( T_H \) to a function left-invariant by \( T_{H,c} \). In particular, in the coordinates \( (z_1, \dots,z_r) \) of \( T_H \),  
\[ u(z_1, \dots, z_r) = u( \log \abs{z_1}^2, \dots, \log \abs{z_r}^2) \] 
As a consequence, \( u \) is also right-invariant by \( T_{H,c} \). Therefore every \( K \)-invariant function on \( G/H \) is automatically \( (K \times T_{H,c}) \)-invariant. The same goes for \( K \)-invariant forms on \( G/H \).  
\end{rmk}

\begin{defn}
A \emph{\( K \)-invariant plurisubharmonic (psh) function} on \( Y \) is a function \( f \) such that for every local \( K \)-equivariant embedding of \( Y \), \( f \) is the restriction of a global \( U_N \)-invariant psh function on \( \Cbb^N \).
\end{defn}

\begin{defn}
Let \( (Y,\xi) \) be a polarized horospherical cone. A \( K \)-invariant \emph{\( \xi \)-radial function} (or \emph{\(\xi\)-conical potential}) on \( Y \) is a positive \( K \)-invariant strictly psh function \( \rho^2 : Y \to \Rbb_{>0} \) such that \( \rho^2|_{G/H} \) is smooth and
\[ \Lcal_{-J \xi} \rho^2 = 2 \rho^2 \] 
on \( Y_{\text{reg}} \). If moreover \( \rho^2 \) is locally bounded on \( Y \), the \( K \)-invariant Kähler (1,1)-current defined as  
\[ \omega := dd^c \rho^2 \] 
is then called a \emph{Kähler cone current}.  
\end{defn} 
Note that \( \omega \) is well-defined by the local theoy of Bedford-Taylor \cite{BT76}. Since \( T_{\xi} \subset T_{H,c} \), a \( K \)-invariant \( \xi \)-radial function is automatically \( \xi \)-invariant.  Before stating the definition of conical Calabi-Yau metrics on \( Y \), let us first make some digressions on linearized line bundles. 

\begin{defn}
Let \( G \) be a linear algebraic group and  \( X \) be any \( G \)-variety. A  \emph{$G$-linearized line bundle} on $X$ is a line bundle \( \Lcal \) over $X$ endowed with an action of $G$ such that (i) the projection $\pi : \Lcal \to X$ is $G$-equivariant and that (ii) \( G \) acts linearly on the fibers of \( \Lcal\).
\end{defn}

 For all section \( s  \in H^0(X, \Lcal) \) and \( x \in X \), \( G \) induces an action on \( H^0(X, \Lcal) \) by: 
\[ (g.s)(x) = g.s(g^{-1} x)  \] 
The group \( G \) then acts on \( H^0(X, \Lcal) \) and the latter is a rational \( G \)-module. For every \textit{reductive} group \( G \) and every line bundle \( \Lcal \) over a \textit{normal} \( G \)-variery, there exists a positive integer \( m > 0 \) such that \( m \Lcal \) is \( G \)-linearized. If in addition \( \Cbb[G] \) is factorial, then one can choose \( m = 1 \). Every \( G \)-linearized line bundle over \( G \) is trivial. We refer the reader to \cite{KKLV89} for the proofs of these assertions. 

\begin{ex}
Let \( H \) be any closed subgroup of \( G \). Every \( G \)-linearized line bundle \( \Lcal \) over \( G/H \) is determined by a character \( \chi_H \) of \( H \). Indeed, given \( \chi_H \), the quotient \( \Lcal_{\chi_H} \)  of \( G \times \Cbb \) by the action:
\[ h(g,z) = (gh^{-1}, \chi_H(h)z) \]
is a  \( G \)-linearized line bundle over \( G/H \). Conversely, if \( \Lcal \) is \( G \)-linearized, then \( H \) acts on the fiber \( \Lcal_{eH} \) by a character \( \chi_{H} \) . We can then directly prove that 
\( G \times \Lcal_{eH} \to \Lcal \),  \( (g,z) \to (gH, g.z) \) induces an isomorphism between \( \Lcal_{\chi_H} \) and \( \Lcal \).
\end{ex}

Let us come back to the setting of our cone \( Y \). 
Since \( Y \) is \( \Qbb \)-Gorenstein, the canonical bundle \( m \Kcal_Y \) is a well-defined Cartier divisor and naturally \( (G \times T_H) \)-linearized for some integer \( m > 0 \). 

Note that the Picard group of \( Y \) is trivial (since \( Y = Y_{\set{0_Y},B} \) by Thm. \ref{bstable_affine_chart}, and \( \text{Pic}(Y_{\set{0_Y},B}) = 0 \) by  \cite[2.1, Proposition]{Bri91}). It follows that there exists a nowhere vanishing holomorphic \( ( G \times T_H) \)-invariant section \( s \) of \( m \Kcal_Y \), and a \( K \)-invariant volume form \( dV_Y \) defined by 
\[ d V_Y = (\sqrt{-1}^{m n^2} s \wedge \ol{s})^{1/m} \]

\begin{rmk} To simplify the notation, we will sometimes abuse the language and say that \( s \) is a ``multivalued'' holomorphic section of \( \Kcal_Y \) and simply write \( dV_Y = i^{n^2} s \wedge \ol{s} \). 
\end{rmk}

The \( \Qbb\)-Gorenstein singularities of \( Y \) implies the existence of a linear function  \( l : (-\Ccal_Y) \to \Rbb_{> 0} \) (which is \( -\beta \) in Prop. \ref{q_gorenstein_condition}) such that for all \( D_\nu \in \Vcal_Y \) and \( D \in \Dcal \):
\[ \sprod{l,\rho(D_\nu)} = -1, \quad \sprod{l, \rho(D)} = -a_D \] 
where \( a_D \) is the coefficient of \( D \) in the expression of \( \Kcal_Y \). 

\begin{defn}  \label{canonical_volume_definition}
A \( T_{H,c} \)-invariant nowhere-vanishing section \( s_Y \in H^0(Y, \Kcal_Y) \) is said to be a \emph{canonical section} if 
\[ \Lcal_{-J\xi} s_Y = \sprod{l, \xi} s_Y \] 
on \( Y_{\text{reg}} \) for all \( \xi \in \text{Int}(-\Ccal_Y) \), where \( - \Ccal_Y \) is the Reeb cone of \( Y \) and \( \Lcal_{-J \xi} \) is the Lie derivative for the right-action of \( -J \xi \).

A volume form \( dV_Y \) on \( Y \) is called \emph{canonical} if on \( Y_{\text{reg}} \),
\[ \Lcal_{-J \xi } dV_Y = 2 \sprod{l,\xi} dV_Y \]
\end{defn}

\begin{rmk} Since every \( \Qbb\)-Gorenstein spherical variety has klt singularities \cite{Pas17}, it follows from \cite{MSY08} (see also \cite[Lemma 6.2]{CS19}) that a canonical volume form always exists. In what follows, we will give an explicit expression of the canonical volume form on \( Y \). 
\end{rmk}

\begin{rmk}
A canonical section \( s_Y \) determines a canonical volume form by setting \( dV_Y := i^{n^2} s_Y \wedge \ol{s_Y} \). Conversely, let \( dV_Y \) be a canonical volume form determined by a multivalued holomorphic section \( s \). For each \( \xi \), there exists a real function \( f_{\xi} \) such that
\[ \Lcal_{-J\xi} s = f_{\xi} s \]
It follows that \( \Lcal_{-J\xi} dV_Y = 2 f_{\xi} i^{n^2} s \wedge \ol{s} \), hence \( f_{\xi} = \sprod{l,\xi} \) for all \( \xi \in -\Ccal_Y \).
\end{rmk}

\begin{rmk} A canonical section is unique up to a constant. Indeed, for any two canonical sections, there exists \( f \in \Cbb[Y]^{*} = \Cbb \) such that \( s_1  = f s_2 \). In particular, a canonical volume form is unique up to a constant.  
\end{rmk}

\begin{defn}
Let \( (Y,\xi) \) be a \( \Qbb\)-Gorenstein polarized horospherical cone of dimension \( n \)  and \( dV_Y \) its canonical volume form. A \emph{\( K \)-invariant conical Calabi-Yau metric} on \( Y \) is a \( K \)-invariant Kähler cone current \( \omega \) on \( Y \) satisfying 
\[ \omega^n = (dd^c \rho^2)^n = dV_Y \] 
such that the \( \xi\)-radial function \( \rho^2 \) is smooth on the regular locus \( Y_{\text{reg}} \). The function \( \rho^2 \) is then said to be a \emph{conical Calabi-Yau potential}. 
\end{defn}

In particular, a Calabi-Yau metric is a singular Kähler-Einstein metric in the sense of \cite{EGZ}. As one will see, the above equation takes the form of a real Monge-Ampère equation due to symmetry by \( K \). 

\subsection{Curvature  and canonical volume form} 

\subsubsection{Curvature form}

The expression of a \( K \)-invariant Kähler form \( \omega \) on \( G/H \) was given by Delcroix \cite{Del20}. Let us first introduce some terminologies. 

Recall that $\tfrak_{H,nc} \simeq \Rbb^r$ is the Lie algebra of $T_H$. Let $\mathfrak{p}^u$ be the Lie algebra of the unipotent radical $P^u$. Denote by \( J \) the complex structure of \( G/H \), and \( \sqrt{-1} \) the complex structure coming from the complexification of the real vector space \( \gfrak/ \hfrak \), which can be decomposed as 
\begin{equation*}
 \gfrak / \hfrak \otimes \Cbb \simeq \bigoplus_{\alpha \in \qfrak^u} \Cbb u_{\alpha} \bigoplus \tfrak_{H,nc} \bigoplus \sqrt{-1} \tfrak_{H,nc}  
\end{equation*}
Here \( u_{\alpha} \) is the eigenvector with weight \( \alpha \) of \( \qfrak^u \). 

Let us come back to our conical embedding \( G/H \hookrightarrow Y \). Consider a  $(G \times T_H)$-linearized line bundle  $\Lcal$ over $G/H$ with a $K$-invariant hermitian metric $h$.  Let $i^{*} \Lcal $ be the line bundle over  $T_H$ induced by the natural inclusion $i : T_H \hookrightarrow G/H$, endowed with the  $T_{H,c}$-invariant metric $i^{*}h$. 

After replacing \( i^{*} \Lcal \) by a multiple large enough, we can suppose that \( i^{*} \Lcal \) is $T_H$-linearized and trivial over $T_H$. In particular, the metric $i^{*}h$ is the curvature form of a \( T_{H,c} \)-invariant global potential \( u \). 

By pullback using the map $\exp : \tfrak_{H,nc} \to T_H$, we can consider $u$ as a function $u: \tfrak_{H,nc} \simeq \Rbb^r \to \Rbb$. This potential in turn defines the metric, as shown by Delcroix. 

\begin{prop} \cite[Proposition 2.7]{Del20}
The metric $h$ is completely determined by the global potential $u : \tfrak_{H,nc} \to \Rbb$ of the metric $i^{*} h$, called the \textit{toric potential}.
\end{prop}

By the Iwasawa decomposition in \textit{loc. cit.}, the $K$-invariant curvature form of $h$, denoted by \( \omega \), is uniquely determined by its restriction to $T_H$. 

Let \( (l_j)_{1 \leq j \leq r} \) be a real basis of \( \tfrak_{H,nc} \), \( (Jl_j)_{1 \leq j \leq r} \) a basis of \( J \tfrak_{H,nc} \),   \( (u_{-\alpha})_{\alpha \in \Phi_{P^u}} \) the basis of the root decomposition in \( \qfrak^u \).  

Consider the associated holomorphic vector 
\[ \zeta_j := \frac{1}{2}(l_j - \sqrt{-1} J l_j) , \quad w_{\alpha} = \frac{1}{2}(u_{-\alpha} - \sqrt{-1} J u_{-\alpha}) \]  
and associate to them the \(T_H\)-invariant holomorphic vector fields of the tangent bundle of \(G/H \) by translation under the action of \( T_H \), i.e. for all \( p \in T_H \), we define: 
 \[ \zfrak_j : p \to  (R_{p^{-1}})_{*} \zeta_j, \quad \wfrak_{\alpha} : p \to (R_{p^{-1}})_{*} w_{\alpha}   \] 

Let $\set{\zfrak^{*}_j}_{1 \leq j \leq r} \cup \set{\wfrak^{*}_{\alpha}}_{\alpha \in \Phi_{P^u}}$ be the dual \( (1,0) \)-forms of the vector fields \( \set{\zfrak_j} \cup \set{\wfrak_{\alpha}} \).

\begin{prop} \cite{Del20} \label{KahlerForm}
For all $x \in \tfrak_{H,nc}$, 
\begin{equation*}
\omega_{\exp(x)H} = \sum_{1 \leq j_1, j_2 \leq r} \frac{1}{4} \frac{\del^2 u}{\del x_{j_1} \del x_{j_2}}  i \zfrak^{*}_{j_1} \wedge \ol{\zfrak_{j_2}^{*}} + \sum_{\alpha \in \Phi_{P^u}} e^{-2 \sprod{\alpha,x}} (\alpha, d_x u + \varpi) \frac{i}{2}  \wfrak^{*}_{\alpha} \wedge \ol{\wfrak_{\alpha}^{*}}  
\end{equation*}
where \( (d_x u + \varpi_Q, \alpha) = \sprod{\nabla u(x) + t_{\varpi_Q}, \alpha} \). In particular, the volume form defined by \( \omega \) is:
\[ \omega^n_{\exp(x)H} = \frac{n!}{2^r} e^{- \sprod{\varpi_Q,x}} \det(d_x^2 u) \prod_{\alpha \in \Phi_{P^u}} (\alpha, d_x u + \varpi_Q)  \bigwedge_{1 \leq j \leq r} \frac{i}{2} \zfrak^{*}_{j} \wedge \ol{\zfrak^{*}_j} \bigwedge_{\alpha \in \Phi_{P^u}} \frac{i}{2} \wfrak^{*}_{\alpha} \wedge \ol{\wfrak^{*}_{\alpha}}  \] 
\end{prop}

\subsubsection{Canonical volume form}

We provide in this part an expression of the canonical volume form of a \( \Qbb \)-Gorenstein horospherical cone.

\begin{prop} \label{CanonicalVolumeForm}
The canonical volume form on $Y$ is \( K \)-invariant on \( G/H \) and writes in $T_H$ as (up to a constant):
\begin{equation*}
(dV_Y)_{\exp(x)H} = e^{\sprod{\beta - \varpi_Q ,x}} \bigwedge_{1 \leq j \leq r} \frac{i}{2} \zfrak^{*}_j \wedge \ol{\zfrak_j^{*}} \bigwedge_{\alpha \in \Phi_{P^u}} \frac{i}{2} \wfrak_{\alpha}^{*} \wedge \ol{\wfrak^{*}_{\alpha}}
\end{equation*}
for all \( x \in \tfrak_{H,nc} \). 
\end{prop}

\begin{proof}
Let \( dV_Y \) be the canonical section of \( Y \) (which always exists, cf. the remark that follows Defn. \ref{canonical_volume_definition}). Consider the following \( K \)-invariant volume form of \( G/H \) as constructed in \cite[Prop 2.6]{Del20b}: 
\[ dV_{\exp(x)H} :=  e^{ \sprod{-\varpi_Q,x}} \bigwedge_{1 \leq j \leq r} \frac{i}{2} \tuple{ \zfrak^{*}_{j} \wedge \ol{\zfrak^{*}_{j}} }_{\exp(x)H} \bigwedge_{\alpha \in \Phi_{P^u}} \frac{i}{2} \tuple{ \wfrak^{*}_{\alpha} \wedge \ol{\wfrak^{*}_{\alpha}}}_{\exp(x)H} \] 
Now for all \( (z_1, \dots, z_r) \in T_H \) and \( \alpha = (\alpha_1, \dots, \alpha_r) \in \Qbb^r \), let \( z^{\alpha} := z_1^{\alpha_1} \dots z_r^{\alpha_r} \) and \( x_j = \log \abs{z_j}^2 \). Consider the following \( K \)-invariant volume form:
\[ (dV'_Y)_{\exp(x)H} := \abs{z}^{2 \beta} dV_{\exp(x)H} \]
Let us prove that \( dV_Y \) restricts to \( dV'_Y \) on \( G/H \). Indeed, since \( dV \) is \( K \)-invariant, \( \Lcal_{-J\xi} dV = 0 \), hence:
\[ \Lcal_{-J\xi} dV'_Y = (\Lcal_{-J\xi} \abs{z}^{2 \beta} ) dV  \]
  By embedding equivariantly \( T_H \) into \((\Cbb^{*})^r \) with coordinates \( (z_1, \dots, z_r) \), one can identify \(-J \xi \) with 
\[ \sum_{j=1}^r \xi_j \Re( z_j \del_{z_j}) ,\; \xi_j > 0 \] 
where \( \xi_j \) is the weight of \( v_{\xi} = (- J \xi - \sqrt{-1} \xi)/2 \) on \( z_j \). By integration, one finds that the one-parameter right-action generated by \( - J\xi \) is 
\[ \phi_{t}^{-J\xi}(z_1, \dots, z_r) = (e^{-t\xi_1} z_1, \dots, e^{-t \xi_r} z_r) \]
A computation on \( T_H \) then yields: 
\begin{align*}
\Lcal_{-J\xi} \abs{z}^{2 \beta } &= \frac{d}{dt} \lvert_{t=0}  e^{-2 t \sum_{j=1}^r \beta_j  \xi_j} \prod_{j=1}^r \abs{z_j}^{2 \beta_j} \\    
&= 2 \sprod{-\beta, \xi} \abs{z}^{2\beta} 
\end{align*}
It follows that \( \Lcal_{-J \xi} dV'_Y = 2 \sprod{l, \xi} dV'_Y \), where \( l = - \beta \) as in Defn. \ref{canonical_volume_definition}. There exists thus a \( K \times (- J \xi) \)-invariant smooth non-vanishing function \( f \) on \( G/H \) such that \( dV_Y \lvert_{G/H} = f dV'_Y \). In particular, \( f\) is also \( \xi \)-invariant for all \( \xi \). Since \( (v_{\xi})_{\xi \in - \Ccal_Y} \) generate the action of \( T_H \), \(f\) is \( T_H \)-invariant and thus descends to a \( K \)-invariant function on the \( K \)-homogeneous manifold \(G/P \), whence \( f \) is constant. The canonical volume form restricted to \( G/H \) is therefore (up to a constant)
\[ (dV_Y)_{\exp(x)H} = \abs{z}^{2\beta} dV_{\exp(x)H} \] 
Since \( \abs{z}^{2 \beta} = e^{\sprod{\beta,x}} \), this completes our proof. 
\end{proof}

\subsubsection{Real Monge-Ampère equation}

Let $G/H \subset Y$ be a \( \Qbb \)-Gorenstein horospherical cone of dimension $n$, with colored cone $(\Ccal_Y,\Dcal_Y)$ and canonical volume form $dV_Y$, as constructed in the previous part. 

The Duistermaat-Heckman polynomial $P_{DH}$ associated to  $G/H$ is defined as 
\begin{equation*}
P_{DH}(p) = \prod_{\alpha \in \Phi_{P^u}} (\alpha, p)
\end{equation*}
for every $p \in \tfrak_{H,nc}^{*}$.  

\begin{lem}
Let \( \xi \) be a Reeb vector and \( \rho \) a \( \xi \)-conical potential on \( Y \). Then $\rho$ is uniquely determined by a smooth and strictly convex positive function $u_{\varpi} = \exp^{*}(\rho|_{T_H}) := u  + \varpi_Q|_{\tfrak_{H,nc}} : \tfrak_{H,nc} \to \Rbb $, satisfying
$$u_{\varpi}(x - t \xi) = e^t u_{\varpi}(x)$$
for all \( t \in \Rbb \). Here, \( u \) is a smooth, strictly convex funtion \( \tfrak_{H,nc}^{*} \to \Rbb \).
We shall call \( u_{\varpi} \) a \textit{\( \xi \)-horospherical cone potential} and \( u \) a \textit{\( \xi \)-toric cone potential}. 
\end{lem}

\begin{proof}
By the Iwasawa decomposition and \( K \)-invariance, \( \rho \) is uniquely determined by its restriction on \( T_H \). Following \cite[Proposition 2.7]{Del20}, for all  $x \in \tfrak_{H,nc}$ :
\begin{equation*}
 \rho^2 |_{T_H} (\exp(x)H)) = u(x) + \varpi_Q (x) (:= u_{\varpi}(x)) 
\end{equation*}
hence $\Lcal_{-J\xi} \rho^2 = 2 \rho^2 $ if and only if $u_{\varpi}(x-t\xi) = e^t u_{\varpi}(x)$. It is clear that $\rho$ is positive and smooth over $G/H$ iff $u_{\varpi}$ is also positive and smooth on $T_H$. In particular, \( u \) is smooth. 

In the local coordinates $z$ of $T_H$, by \( T_{H,c} \)-invariance,  we have 
\[\rho^2 (z) = \rho^2 (\log \abs{z_1}^2, \dots, \log \abs{z_r}^2)\]
Since \( \rho^2 \) is smooth and strictly psh on \( T_H \), a direct computation of the Levi form shows that the function $u_{\varpi} = \rho^2 (\exp(.)H)$ is strictly convex. In particular, \( u \) is also strictly convex since \( \varpi_Q \) is linear.  
\end{proof}

\begin{prop}
Let \( \omega \) be a \( K \)-invariant Kähler cone current defined by a horospherical cone potential \( u_{\varpi} = u + \varpi_Q \) on \( \tfrak_{H,nc}^{*} \). 
The Calabi-Yau equation on \( G/H \) is equivalent to the real Monge-Ampère equation
\begin{equation} \label{RealMongeAmpere}
\det(d^2 u) \prod_{\alpha \in \Phi_{P^u}} (\alpha, d u + \varpi_Q) = C e^{\beta }
\end{equation}
where \( C \) is a constant depending only on the dimension $n$. Moreover, $u$ is a solution of (\ref{RealMongeAmpere}) iff $u_{\varpi} = u + \varpi_Q$ is a solution of
\begin{equation} \label{RealMongeAmpere2}
\det(d^2 u_{\varpi}) P(d u_{\varpi}) = e^{\beta}
\end{equation}
\end{prop}

\begin{rmk}
By rescaling of variables and using homogeneity of \(P_{DH}\), we see that solving (\ref{RealMongeAmpere2}) is equivalent to solving 
\[ \det(d^2 v) P_{DH}(d v) = C e^{\lambda \sprod{\beta,x}} \] 
for all real positive \( \lambda \) and some positive constant \( C > 0 \) depending only on \( \lambda \).
\end{rmk}

\begin{proof}
Indeed, using the expressions of the curvature and volume forms as in Prop. \ref{KahlerForm} and \ref{CanonicalVolumeForm}, we have 
\begin{equation*}
\omega^n = dV_Y \iff  n! \det(d^2 u) \prod_{\alpha \in \Phi_{P^u}}(\alpha, d u + \varpi_Q) = 2^{r} e^{\beta}
\end{equation*}
Since the real Monge-Ampère operator (in Alexandrov's sense) is invariant under translation by an affine function, we obtain the equivalence between ($\ref{RealMongeAmpere}$) for $u$ and equation ($\ref{RealMongeAmpere2}$) for $u_{\varpi}$ up to a positive constant \( C \). Since \( C \) depends only on \( G/H \), rescaling by a positive constant and using homogeneity of \( P_{DH} \), we can suppose that \( C = 1 \). 
\end{proof}


\subsection{Proof of Theorem \ref{main_theorem_volume_minimization}}
Let \( (\Ccal_Y, \Dcal_Y) \) be the colored cone of \( Y \). Since the support of \( \Ccal_Y \) contains \( \rho(\Dcal) \), the dual cone \( \Ccal_Y^{\vee} \) is then contained in 
\[ \set{p \in \Mcal_{\Rbb}, \sprod{p, \rho(\Dcal)}} \geq 0 \] 
It follows that 
\[ P_{DH}|_{\Ccal_Y^{\vee}} \geq 0, \quad  P_{DH}|_{\text{Int}(\Ccal^{\vee}_Y)} > 0 \]
Recall that \( \beta \) is the linear function on \( \Ccal_Y \) corresponding to the \( \Qbb \)-Cartier divisor \( \Kcal_Y \). In what follows, \( d \lambda \) will denote the Lebesgue measure on an appropriate affine space (which will be clear in the context). 

\begin{rmk}
To simplify our exposition and to avoid confusion due to \( T_H \) acting in two different ways, we will work on the cone \( \Ccal_Y \) of \textit{opposite Reeb vectors} and consider the \textit{normalized opposite Reeb vectors} defined by 
\[ \sprod{\beta, \xi_{-} } = 1, \; \xi_{-} \in \Ccal_Y \]
Note that the conical condition of \( u_{\varpi} \) under \( - J \xi \), where \( \xi \in -\Ccal_Y \), then becomes 
\[ u_{\varpi} (x + t \xi_{-})=e^t u_{\varpi} (x), \; \text{for} \; \xi_{-} \in \Ccal_Y \] 
\end{rmk}
Let \( \sprod{.,.} \) be the natural bilinear form on \( \Mcal \times \Ncal \), and \(( \delta_1, \dots, \delta_{r-1}, \beta) \) a basis of \( \Mcal \) such that : 
\[ \sprod{\delta_j, \xi_{-}} = 0, \; 1 \leq j \leq r-1\] 
We will denote as in Hamiltonian mechanics the coordinates of \( (\Ncal,\Mcal) \) by \( (x,p) \) respectively. 
Consider the polytopes
\[\Delta_{\xi} = \set{p \in \Ccal^{\vee}_Y ,  \sprod{p,\xi_{-} } = 1} , \quad \Delta_{\xi}^0 := \Delta_{\xi} - \beta
\]
with Duistermaat-Heckman barycenters 
$\bar_{DH}( \Delta_{\xi})$ and $\bar_{DH}( \Delta^0_{\xi})$
Let
\[ \phi_{ \Delta^{0}_{\xi}}(x) := \sup_{p \in \Delta_{\xi}^0} \sprod{x,p}, \quad \phi_{\Delta_{\xi}} = \phi_{\Delta^0_{\xi}} + \beta \] 
be the support function of the polytopes \( \Delta_{\xi}^0 \) and \( \Delta_{\xi} \).

\begin{defn}The \( P_{DH} \)-weighted volume function, or \textit{Duistermaat-Heckman volume}, is defined over the normalized opposite Reeb vectors as: 
$$\xi_{-} \to \vol_{DH} (\xi_{-} ) := \vol_{DH} (\Ccal^{\vee}_Y \cap \set{\sprod{\xi_{-},. } \leq 1})$$
\end{defn}

\begin{lem}
The \textit{\( P_{DH} \)-weighted volume function} has the following global expression: 
\[ n! \vol_{DH} (\xi_{-} ) =  \int_{\Ccal^{\vee}_Y} e^{- \sprod{\xi_{-} ,p}} P_{DH} (p) d \lambda (p), \]
In particular, \( \vol_{DH} \) is smooth and strictly convex, and its minimum is attained in \( \text{Int}(\Ccal_Y) \).
\end{lem}

\begin{proof}
Let \( g := P_{DH} \) and \( \Ccal := \Ccal_Y \). By the \((n-r)\)-homogeneity of \( g\):  
\[ \int_{\set{\sprod{\xi_{-}, .} <  s}} 1_{\Ccal^{\vee}} g(p)d \lambda (p) = s^{n} \int_{\set{\sprod{\xi_{-},.} <  1}} 1_{\Ccal^{\vee}} g(p) d\lambda (p) \] 
From this, we have: 
\begin{align*}
 \int_{\Ccal^{\vee}} e^{-\sprod{\xi_{-},p}} g(p) d \lambda (p) &= \int_{\set{s > 0}} e^{-s} \frac{d}{ds} \tuple{ \int_{\set{\sprod{\xi_{-},.} < s}} 1_{\Ccal^{\vee}} g(p) d \lambda (p)} ds  \\
 &= \int_{\set{s > 0}} n e^{-s} s^{n-1} \vol_g(\xi_{-}) ds = n! \vol_g(\xi_{-}) 
\end{align*} 
It follows that \( \vol_g \) is continuous and convex as the average by the positive measure \( g d \lambda \) of the continuous and convex function \( p \to e^{-\sprod{p,\xi_{-}}}  \). To show that \( \vol_g \) attains its minimum in the interior of the cone, it is enough to show that \( \vol_g (\xi_{-} ) \to \infty \) when \( \xi_{-} \to \xi_0 \in \del \Ccal \). Indeed, since \( \xi_0 \in  (\Ccal^{*})^{\perp} \) and that  \( \xi_{-} \) is normalized, there exist \( p_0, p_1 \in \Ccal^{*} \) such that \( \sprod{\xi_0, p_0} = 0, \sprod{\xi_0,p_1} = 1 \). It follows that the polyhedron
\( P_{\xi_0} := \set{\set{\xi_0,.} = 1} \)
is not bounded since it contains \( p_1 + c p_0, \; \forall c > 0 \), hence \( \vol_g(\xi_0) \geq \int_{P_{\xi_0}} g d \lambda   = \infty \), which finishes our proof. 
\end{proof}

Since we want the conical condition \( u_{\varpi} (x+t \xi_{-}) = e^t u_{\varpi} (x) \) for all \( \xi_{-} \in \Ccal_Y \), it is natural to impose that \( d u_{\varpi}(\Rbb^r) = \text{Int}(\Ccal_Y^{\vee}) \), or a relatively weaker one as follows. 

\begin{prop} \label{prop_volume_minimization_equivalent_existence}
Let \( \xi \in -\Ccal_Y \) be a Reeb vector and \( \xi_{-} \) its opposite. The following assertions are equivalent
\begin{enumerate}
    \item  There exists uniquely up to translation a strictly convex smooth function $u: \Rbb^r \to \Rbb$ that satisfies 
\( u_{\varpi} (x + t \xi_{-} ) = e^{t} u_{\varpi}(x) \) and
    \begin{equation} \label{conical_CY}
    P_{DH}(d u_{\varpi} ) \det( d^2 u_{\varpi}) = e^{\beta }, \; \overline{d u_{\varpi}(\Rbb^r)} = \Ccal^{\vee}_Y, 
    \end{equation}
    \item $\bar_{DH}(\Delta_{\xi}) = \beta$.
    \item  The normalized opposite Reeb vector \( \xi_{-} \) is the unique minimizer of the volume \( \vol_{DH} \). 
\end{enumerate}
\end{prop}

\begin{proof}
Let \( v := u_{\varpi}, g := P_{DH}\) and \(\Ccal := \Ccal_Y \). Set \( v = e^{\psi} = e^{\phi + \beta} \). The asymptotic condition \( \ol{d v (\Rbb^r)} = \Ccal^{\vee} \) and the fact that \( v(x+t\xi_{-} ) =e^t v(x) \) gives \(d \psi(\Rbb^r) \subset \text{Int}(\Ccal^{\vee}) \) and \( \sprod{d \psi(.), \xi_{-} } = 1 \), hence
\[ \ol{d \psi(\Rbb^r)} = \Delta_{\xi}, \quad \ol{ d \phi(\Rbb^{r-1})} = \Delta_{\xi}^0 \]
It is immediate to check that the function \( \phi \) is strictly convex, smooth and satisfies as a function on \( \Rbb^r \): 
\begin{equation*}
\phi(x + t \xi_{-} ) = \phi(x) + t - t \sprod{\beta, \xi_{-}} = \phi(x)
\end{equation*}
by the normalization condition on \( \xi_{-} \). It follows that \(\phi\) depends only on the coordinates \( \delta\). 
A straightforward but tedious computation gives us: 
\begin{equation*}
\det(d^2 v) = \det( d_{\delta, \beta }^2 (e^{\phi + \beta}) ) =  e^{(r \beta + r \phi)}  \det(d_{\delta}^2 \phi)  
\end{equation*}
and
\begin{align*}
g(d v ) &= g( d_{\delta, \beta}  e^{\phi + \beta}) \\
&= e^{(n-r)(\phi + \beta)} g( d_{\delta} \phi + \beta )
\end{align*}
We then obtain the following equivalence under the conical condition
\begin{align*}
\det(d^2 v ) g(dv) &= e^{ n \beta }, \quad \ol{dv(\Rbb^r)} = \Ccal^{\vee} \iff \\
\det( d^2_{\delta} \phi) g( d_{\delta} \phi + \beta) & = e^{- n \phi}, \quad \ol{d \phi(\Rbb^{r-1})} = \Delta_{\xi}^0
\end{align*}
Since \(\bar_g(\Delta_{\xi}^0) = 0 \) iff \( \bar_g(\Delta_{\xi}) = \beta\), the equivalence between \( (i) \) and \( (ii) \) is a direct consequence of Thm. \ref{realMA_barycenter}. 

For the equivalence between \( (ii) \) and \( (iii) \), first remark that: 
\[ -d_{\xi} \log (n! \vol_g(\xi_{-})) = \frac{\int_{\Ccal^{\vee}}  p e^{-\sprod{\xi_{-},p}} g(p)d \lambda (p) }{\int_{\Ccal^{\vee}} e^{-\sprod{\xi_{-}, p}} g(p) d \lambda (p) }  =  \bar_g(\Delta_{\xi}) \]
 Indeed: 
\begin{align*}
 \int_{\Ccal^{\vee}} p e^{-\sprod{\xi_{-},p}} g(p) d \lambda (p) &= \int_{\Ccal^{\vee} \cap \sprod{\sprod{\xi_{-},.} = s} } \int_{\set{s > 0}} p e^{-s} g(p) d \lambda  (p) ds  \\ 
 &= \int_{\Ccal^{\vee} \cap \sprod{\sprod{\xi_{-},.} = 1}} p g(p) d \lambda (p) \int_{\set{s > 0}} s^n e^{-s} ds  
\end{align*}
By the same token, 
\[ \int_{\Ccal^{\vee}} e^{-\sprod{\xi_{-},p}} g(p) d \lambda (p) = \int_{\Ccal^{\vee} \cap \set{\sprod{\xi_{-},.} = 1}} g(p) d \lambda (p) \int_{\set{s > 0}} s^n e^{-s} ds \]  
Since \( \xi_{-} \to - \log n! \vol_g (\xi_{-}) \) is a convex function, going to \( \infty \) when \( \xi_{-} \to \xi_0 \in \del \Ccal \), its unique minimizer belongs to \( \Ccal \) and coincides with the critical point. As a consequence, \( \xi_{-} \) is the unique minimizer of \( \vol_g \) if and only if \( \bar_g(\Delta_{\xi}) = \beta\). 
\end{proof}

\begin{prop} \label{prop_existence_open_orbit_equivalent_global_existence}
The following are equivalent:
\begin{itemize}
    \item[1)] The polarized horospherical cone \( (Y,\xi) \) admits a unique \( K \)-invariant conical Calabi-Yau metric.
    \item[2)]  There exists uniquely up to translation a strictly convex smooth function $u : \Rbb^r \to \Rbb$ that satisfies 
\( u_{\varpi} (x + t \xi_{-} ) = e^{t} u_{\varpi}(x) \) and
    $$P_{DH}(d u_{\varpi} ) \det( d^2 u_{\varpi}) = e^{\beta }, \; \overline{d u_{\varpi}(\Rbb^r)} = \Ccal^{\vee}_Y, $$
\end{itemize}
\end{prop}

\begin{proof}

The implication \((1) \Rightarrow (2) \) is obvious. 

$(2) \Rightarrow (1)$. From the assumption, the function \( u_{\varpi} \) corresponds to a smooth conical Calabi-Yau potential over \( G/H \). Let us show that it can be extended to a locally bounded function over \( Y \). First, by Prop. \ref{prop_spherical_dominated_by_toroidal} there exists a proper and birational morphism \( d \) and a toroidal horospherical variety \( \wt{Y} \) such that
$$d : \wt{Y} \to Y$$ 
The torus $T_H$ is identified under pullback by this morphism to the open dense orbit of a toric variety $Z$ whose fan $\Ccal_Z$ (in the toric sense) corresponds exactly to the opposite of $\Ccal_Y$ with its colors removed. In particular $Z$ is a toric cone by our criterion \ref{spherical_cone_criterion}. The toric cone potential $u$ then corresponds to a potential, still denoted by $u$, over the open dense orbit $T_H \subset Z$.  

By the \( C^0\)-estimate in Thm. \ref{realMA_barycenter}, one obtains:
\begin{equation*}
u_{\varpi} \leq \exp ( \phi_{ \Delta_{\xi}} ) = \exp \tuple{ \phi_{\Delta^0_{\xi}} + \beta},
\end{equation*}
It is then enough to show that the function $\exp(\sup_{p \in \Delta_{\xi}^0} \sprod{x,p})$ pulled back by 
\( \text{Log}: T_H \to t_{H,nc} \) 
extends to a continuous function over $Z$. This is well-known in the toric situation (cf. e.g. \cite[Proposition 3.3]{BB13}, \cite[3.3]{Ber20}). We provide a proof here for the reader's convenience. 
The toric variety \( Z \) is endowed with an action of \(\Rbb^{*}_{+} \) generated by \( - J \xi \). Let \( (n_i)_{1 \leq i \leq N} \) be a collection of lattice points of \( \Ccal_Z^{\vee} \), satisfying \( \sprod{n_i, \xi} \leq  C \), where \( C \) is a positive constant. Consider an action of \(\Rbb_{+}^{*} \) on \( \Cbb^{N}_{Z_1, \dots Z_N} \) defined by: 
\[ c. Z_i := c^{- \sprod{\xi, n_i}} Z_i \] 
Since \( Z \) is a \( T_H \)-toric affine variety, for \( C \) sufficiently large, we have a \( T_H \)-equivariant (hence \( \Rbb^{*}_{+} \)-equivariant) embedding: 
\[ F_p : T_H \simeq (\Cbb^{*})^r \to \Cbb^{N}, z \to (z^{n_1}, \dots, z^{n_N} )  \]
such that the closure of its image coincides with \( Z \).  Now by Kakutani's thereom, there exists a finite number of elements \( p_1, \dots, p_M \), which are the vertices of \( \Delta_{\xi}^0 \), satisfying
\[ \exp \tuple{ \sup_{p \in \Delta_{\xi}^0} \sprod{x,p} } = \max_{1 \leq i \leq M} \exp \sprod{x,p_i} \] 
It is then enough to show that for all \( p_i \), \( \text{Log}^{*} \exp \sprod{x, p_i} \) extends continuously over \( Z \). We remove the index \( i \) in what follows. Since \( p \in \Mcal(T_H) \cap \Ccal_Z^{\vee} \) and that \( \Ccal_Z^{\vee} \) is of maximal dimension, there exist lattice points \( n_{i_1}, \dots, n_{i_L}  \in \Mcal(T_H) \cap \Ccal_Z^{\vee} \) such that \( p = a_{1} n_{i_1} + \dots + a_{L} n_{i_L} \).  Under push-forward by \( F_p \), the function \( \text{Log}^{*} \exp \sprod{x,p} \) then corresponds to: 
\[ \abs{Z_{i_1}}^{2a_{1}} \dots \abs{Z_{i_L}}^{2a_{L}} \] 
which is clearly continuous on \( \Cbb^N \).
As a consequence, \( \exp \phi_{\Delta_{\xi}^0} \) extends to a continuous function on \( Z \).  Finally, we obtain a locally bounded (in fact continuous) conical Calabi-Yau potential \( u_{\varpi} \) over \( Y \). This completes our proof of the \( (2) \Rightarrow (1) \) direction, except the smoothness of the radial function on the regular locus of $Y$. 
The smoothness over the regular locus results from \cite{N22}. 
\end{proof}

\begin{proof}[Proof of main Theorem \ref{main_theorem_volume_minimization}]
We shall prove that \((i) \Ra (ii) \Ra (iii) \) and \((i) \Lra (iii)\). 

It is known that existence of a \( \xi \)-conical Calabi-Yau metric implies \( K \)-stability for \( (Y,T_H,\xi) \) (cf. \cite[Corollary A.4]{LWX}).  From \cite[Theorem 6.1]{CS18} (see also \cite[Remark 2.27]{LX18}) if \( (Y,T_H,\xi) \) is \( K \)-stable, then \( \xi \) minimizes the volume functional on the space of normalized Reeb vectors (i.e. \(\xi_{-}\) is a minimizer of the \( P_{DH} \)-weighted volume functional in our case). 

The equivalence between conical Calabi-Yau metrics existence and volume minimization is a direct consequence of Prop. \ref{prop_volume_minimization_equivalent_existence} and Prop.  \ref{prop_existence_open_orbit_equivalent_global_existence}. 
\end{proof}

\subsection{Variational approach on weighted real Monge-Ampère equations}
\subsubsection{Existence condition on a Kähler-Einstein-related equation}
We have shown that the study of  (\ref{RealMongeAmpere2}) is equivalent to solving an equation of the type: 
\begin{equation} \label{KE_MA}
g(d\phi + \beta ) \MA_{\Rbb} (\phi) = e^{-n\phi}, \quad \ol{ d \phi(\Rbb^{r-1})} = \Delta^0_{\xi}
\end{equation}
where \( \MA_{\Rbb} \) denotes the real Monge-Ampère operator in the sense of Alexandrov. It is clear that \( g \geq 0 \) on \( \Delta_{\xi}^0 \) and \( g > 0 \) on \( \text{Int}(\Delta_{\xi}^0) \). 

In this part, we will study (\ref{KE_MA}) in the more general setting where \( \Delta \) is any \textit{convex body of dimension m} and \( g \) a \textit{bounded} function satisfying.
\[ g|_{\Delta} \geq 0, \quad g|_{\text{Int}(\Delta)} > 0 \]

\begin{thm} \label{realMA_barycenter}
Let \( \Delta \) be a convex body in \( \Rbb^{m} , m \in \Nbb^{*} \) such that \( 0 \in \text{Int}(\Delta) \). Let \( g \) be a bounded function, \( \geq 0 \) on \( \Delta \) and \( > 0 \) on \( \text{Int}(\Delta) \). Then there exists a strictly convex function \( \phi \), unique up to translation,  satisfying :
\begin{equation} \label{KE_MA_general} 
g(d  \phi ) \MA_{\Rbb} ( \phi ) = e^{-\phi}, \; \ol{d \phi(\Rbb^{m})} = \Delta 
\end{equation}
in the weak sense if and only if \(\bar_g(\Delta) = 0 \). 
Moreover, there exists a constant \( C > 0 \) such that 
\[ \sup_{\Rbb^m} \abs{\phi - \phi_{\Delta}} \leq C \] 
If \( g \) is smooth, then the weak solution \( \phi \) is actually a smooth solution. 
\end{thm}

\begin{proof}
In \cite{BB13}, the weight \( g\) is supposed to be \( > 0 \) \textit{everywhere}, but the proof can be adapted almost word to word from \cite[Theorem 1.1]{BB13} to our setting, \textit{except the regularity of \( \phi \)}, as remarked after formula (2.1) in \textit{ibid.}. Instead of repeating the authors' proof, we indicate below the places where the arguments involving \( g \) are used, and show that our assumptions on \( g\) do not affect the validity of the theorem. For regularity of the solution when \( g \) is smooth and the \( C^0\)-estimate, see Lem. \ref{lemma_smooth_weight_implies_smooth_solution} and Lem. \ref{unicity_and_C_estimate}.   

Let \( \phi : \Rbb^m \to [-\infty, \infty[ \) be a convex proper function (i.e. non identically equal to \( - \infty \) ). The \textit{subdifferential} \( \del \phi \) of \( \phi \) is defined as  : 
\[ \del \phi(x_0) := \set{z \in \Rbb^m, \; \phi(x_0) + \sprod{z, x - x_0 } \leq \phi(x), \; \forall x \in \Rbb^m } \]
Denote by 
\( \Pcal(\Rbb^m) \) and \(\Pcal^{+}(\Rbb^m) \) the set of proper convex functions \( \phi \) in \( \Rbb^m\) such that \( \del \phi(\Rbb^m) \subset \Delta \) and \( \del \phi(\Rbb^m) = \text{Int}(\Delta) \), respectively. 
\begin{itemize}
    \item The \textit{\( g \)-weighted Monge-Ampère measure} in Alexandrov's sense is again well-defined:
for every Borel subset \( E \subset \Rbb^m\) and convex proper function \( \phi \):
    \[ \int_E \MA_g(\phi) = \int_{\del \phi(E)} g(p) d \lambda(p)\] 
     This still makes sense even if \( g \) is any \( L^1(\Rbb^m) \) function. 
    \item The measure \( g d\lambda \) still defines a positive finite Borel measure on \( \Delta \) because \( g\) vanishes only on \( \del \Delta \) and that \( d \lambda \) does not put mass on \( \del \Delta \). In particular, for all \( \phi \in \Pcal^{+}(\Rbb^m)\),
    \[ \int_{\Rbb^m} \MA_g(\phi) = \int_{\text{Int}(\Delta)} g d \lambda = \vol_{g}(\Delta) > 0\] 
    \item There exists a unique functional \( \Ecal_{g} \), called \textit{\(g\)-energy}, defined on \( \Pcal^{+}(\Rbb^m) \) such that 
    \[ d_{\phi} \Ecal_{g} = \MA_g(\phi) \]
    The proof is the same as in \cite[Proposition 2.9]{BB13}. A fact used by the authors in the proof is the following: for all non-negative continuous function \( f \),
    \[ \int_{\Rbb^m} f \MA_g(\phi) = \int_{\Delta} f( d \phi^{*}|_p) g(p) d \lambda(p) \]
    where \(\phi^{*} \) is the Legendre transform of \( \phi \). The equality still makes sense in our setting since \(d \phi^{*} \) exists \( d \lambda\)-a.e., hence \( g d \lambda \)-a.e..
    \item \( \Ecal_{g} \) extends uniquely to an increasing u.s.c. functional  \( \Pcal(\Rbb^m) \to [-\infty, +\infty[ \) by: 
\begin{equation} \label{g_energy_expression} \Ecal_g (\phi) = -m! \int_{\Delta} \phi^{*}(p) g(p) d\lambda(p) 
\end{equation}
The hypothesis \( g \geq 0 \) on \( \Delta \) guarantees that  \( \Ecal_g \) is an increasing and u.s.c functional on \( \Pcal(\Rbb^m) \). Indeed, since the Legendre transform is decreasing and \( g \geq 0 \), we see that \( \Ecal_g \) is increasing. The upper-semicontinuity now follows from Fatou's lemma (which applies to the positive measure \( g d \lambda \)) and the fact that for every sequence \( \phi_i \to \phi \) in \( \Pcal(\Rbb^m) \), \( \inf \phi_i^{*} \geq \phi_i^{*} \).
\end{itemize} 

Let \( E^1(\Rbb^m) \) be the set of \( \phi \in \Pcal(\Rbb^m) \) whose \(g\)-Monge-Ampère mass is \( \vol_g(\Delta) \) and whose \( g\)-energy is finite. Notice that \( \Pcal^{+}(\Rbb^m) \subset E^1(\Rbb^m) \).  Let 
\[ \Ical(\phi) := -\log \int e^{-\phi(x)} d \lambda(x) \in ]-\infty, + \infty] \] 
The \textit{Ding functional} \( \Dcal_{g} : E^1 \to ]-\infty,+\infty] \) is defined as 
\[ \Dcal_{g, \Delta}(\phi) = \frac{1}{r! \vol_g(\Delta)} \Ecal_g(\phi) - \Ical(\phi) \] 
The \textit{geodesic} between two strictly convex and smooth functions \( \phi_0 \) and \( \phi_1 \) in \( \Pcal(\Rbb^m) \) is the function \( t \in [0,1] \to \phi_t = (t \phi_0 + (1-t)\phi_1)^{*} \in \Pcal^{+}(\Rbb^m) \).

The following properties of \( \Dcal_{g} \) still hold in our setting: 
\begin{itemize}
    \item \( \Dcal_{g} \) is \textit{strictly concave} along geodesics in \( \Pcal^{+}(\Rbb^m) \), i.e. \( t \to \Dcal_g(\phi_t) \) is a concave function which is affine iff the geodesic is a translation \( \phi_t(p) = \phi(p+ta) \) \cite[Proposition 2.15]{BB13}.  The proof uses the Prekopa-Leindler inequality on \(t \to \Ical(\phi_t) \), which is independent of \( g\), and the affineness of \( \Ecal_{g} \) along a geodesic (this is clear by equation (\ref{g_energy_expression}))
    \item The \textit{coercivity} of \( \Dcal_{g} \) \cite[Theorem 2.16]{BB13}: for all \( \varepsilon > 0 \), there exists \( C_{\varepsilon} > 0 \) such that 
    \[\Dcal_{g} \leq (1 - \varepsilon) \Ecal_{g} +  C_{\varepsilon} \]  
     \cite[Lemma 2.14, Theorem 2.16]{BB13}. Here, the proof uses concavity of \( \Dcal_g \) as well as affineness of \( \Ecal_{g} \) along geodesics, and refinement by scaling of the functional \( - \Ical \) (which is independent of \(g\)). The latter is based on the following inequality \cite[Formula (2.16)]{BB13}: for a well-chosen \( \phi_0 \in \Pcal(\Rbb^m) \),
     \[ e^{-\phi_0} \leq C \MA_g (\phi_0) \]
     \item \( \Dcal_{g} \) is bounded iff \( \Dcal_{g}(\phi(.+a)) = \Dcal_{g}(\phi) \) iff \( \bar_{g}(\Delta) = 0 \). Again, concavity of \( \Dcal_g\) and affineness along geodesics of \( \Ecal_g \) are used in the arguments. The proof for the direction ``\(\bar_{g}(\Delta) = 0 \Ra \Dcal_g\) bounded'' uses the coercivity of \( \Dcal_g \). 
\end{itemize}
\textit{Conclusion:} Suppose that \( \bar_g(\Delta) = 0 \), then \( \sup_{\phi \in E^1} \Dcal_g \) is attained. This can be seen by showing that \( \Dcal_g \) is u.s.c. But since \( \Ecal_g \) is already u.s.c. (the non-negativity of \( g \) is required for this property), it is enough to show that \( -\Ical \) is u.s.c.. This functional is actually continuous and the proof is the same as in \cite{BB13} because the expression of \(-\Ical\) does not involve \( g \). The maximizer of \( \Dcal_g \) actually satisfies (\ref{KE_MA_general}) in the weak sense, and is unique up to translation by strict concavity of \( \Dcal_g \).

Conversely, if there exists a maximizer \( \phi \) satisfying (\ref{KE_MA_general}) in the weak sense, then \( \Dcal_g \) is bounded, hence \( \bar_g(\Delta) = 0 \). 

Finally, apply Lem. \ref{unicity_and_C_estimate} for \( \mu = e^{-\phi} d \lambda \) yields the \( C^0\)-estimate. The assertion on smoothness is proved in Lem. \ref{lemma_smooth_weight_implies_smooth_solution}, which completes our proof of the theorem. 
\end{proof}

Let us recall the following regularity properties, which go back to Caffarelli \cite{Caf90a, Caf90b} (see also \cite[Theorem 2.24]{BB13}): 

\begin{thm} \label{regularity_theorem}
Let \( \Omega \) be an open bounded convex set in \( \Rbb^m \) and \( f > 0 \). Then every convex solution \( \phi \) on \( \Omega \) of the equation: 
\[ \MA(\phi) = f d\lambda(x), \quad \phi = 0 \; \text{on} \; \del \Omega \] 
is 
\begin{enumerate}
\item  strictly convex and \( C^{1,\alpha}_{\text{loc}}(\Omega) \) for every \( \alpha > 0 \) if there exists a constant \( C > 0 \) such that \( 1/C \leq f \leq C \). 
\item  \( W^{2,p}_{\text{loc}}(\Omega) \) for all \( p > 1 \) if \( f \) is in addition continuous.  
\end{enumerate} 
\end{thm}

\begin{lem} \label{lemma_smooth_weight_implies_smooth_solution}
Let \( \phi \in \Pcal^{+}(\Rbb^m) \) be a weak solution of \( \MA_g(\phi) = e^{-\phi} d \lambda \). If \( g \) is smooth, then \( \phi \) is actually smooth. 
\end{lem}

\begin{proof}
Since \( \phi \) is proper, \( \Omega_R := \set{\phi < R } \Subset \Rbb^m \) are bounded convex and relatively compact domains in \( \Rbb^m \) which cover \( \Rbb^m \). In particular, \( e^{-\phi}, e^{\phi} \in L^{\infty}(B(0,R)) \) for all \( R > 0 \). By assumption \( g > 0 \) on \( \text{Int}(\Delta) \), hence :
\[ C_g := \sup_{\Delta} g  > 0 \] 
It follows that for each Borel set \( E \), 
\begin{align*}
C_g \int_{E} \MA_{\Rbb}(\phi) &= C_g \int_{\del \phi(E)} d \lambda(p) \geq \int_{\del \phi(E)} g(p) d \lambda(p) \\
&= \int_E e^{-\phi} d \lambda(x) \geq e^{-R} \int_E d \lambda(x) 
\end{align*}
hence \( \MA_{\Rbb} (\phi) \geq (e^{-R}/ C_g ) d \lambda(x) =: C_{R,g} d \lambda(x) \) on \( \Omega_R \) in the sense of measures. 

We assert that \( \del \phi (\ol{\Omega}_R) \) is compact in \( \text{Int}(\Delta) \)
for all \( R \). Indeed, by definition of subdifferential
\[ \del \phi(\ol{\Omega}_R) = \cup_{x \in \ol{\Omega}_R} \del \phi(x) \subset \del \phi(\Rbb^m) = \text{Int}(\Delta) \]  
Since \( \ol{\Omega}_R \) is compact, \( \del \phi(\ol{\Omega}_R) \) is also compact (for a proof, see e.g. \cite[Lemma A.22]{Fig}), so it is a compact of \( \text{Int}(\Delta) \). By our assumption \( g > 0 \) on \( \text{Int}(\Delta) \) and bounded on \( \Delta \), so we have:
\[ \inf_{\del \phi(\Omega_R)} g \geq \min_{\del  \phi(\ol{\Omega}_R)} g > 0 \] 
As a consequence, there exists a constant \( C_{\Delta,R } = 1 / \inf_{\del \phi (\Omega_R)} g > 0 \) (depending only on \( \Delta \) and \( R \)) such that : 
\[ \MA_{\Rbb}(\phi) \leq C_{\Delta,R} d \lambda(x) \]
on \( \Omega_R \) in the sense of measures. It follows that \[ C_{R,g} d \lambda(x) \leq \MA_{\Rbb} (\phi) \leq C_{\Delta,R} d \lambda(x) \] 
By Caffarelli's regularity results in Thm. \ref{regularity_theorem}, \( \phi \in C^{1,\alpha}(\Omega_R) \) for all \( R > 0 \), hence \( d \phi \in C^0(\Rbb^m) \). Therefore : 
\[ \MA_{\Rbb}( \phi) = h d \lambda\] 
for the continuous function \( h = e^{-\phi}/g(d \phi) \) on \( \Omega_R \) satisfying \( C_{R,g} \leq h \leq C_{\Delta,R} \). Finally, \( \phi \in W^{2,p}_{\text{loc}}(\Rbb^m) \) for all \( p > 0 \). From Evans-Krylov theory \cite{E82}, \cite{K83}, we obtain a \( C^{2,\alpha} \)-estimate of \( \phi \) on \( \Omega_R \). Finally, a linear bootstrapping argument with Schauder estimate (cf. \cite[proof of Lemma 7.2]{E82}), and smoothness of \( g \) allow us to conclude that \( \phi \in C^{\infty}(\Omega_R) \) for all \( R > 0 \), hence \( \phi \in C^{\infty}(\Rbb^m) \).  
\end{proof}

\begin{lem} \label{unicity_and_C_estimate}
Let \( \Delta \) be a convex body in \( \Rbb^m \) and \( \mu \) a measure of total mass \( \vol_g(\Delta) \) on \( \Rbb^m\). Suppose that there exists a weak convex solution \( \phi: \Rbb^m \to \Rbb\) to the following problem:
\begin{equation} 
\MA_g( \phi) = \mu, \quad  (\del \phi)(\Rbb^m) = \text{Int}(\Delta) 
\end{equation}
Then after normalizing \( \phi \) such that \( \sup_{\Delta}(\phi - \phi_{\Delta}) = 0\), we have for all \( q > m \) the following inequality: 
\[ \sup_{\Rbb^m} \abs{\phi - \phi_{\Delta}} \leq C(q,m, \Delta, \mu) \] 
where \( C(q,m,\Delta,\mu) \) is a constant that depends only on \( q, m, \Delta \) and the measure \( \mu \). More precisely:
\[ C(q,m,\Delta, \mu) = \frac{d(\Delta)}{\vol(\Delta)} \int_{\Rbb^m} \abs{x} \mu + C_{q,m} \frac{d(\Delta)^{(1  - m/q)}}{\vol (\Delta)} \tuple{ \int_{\Rbb^m} \abs{x}^q \mu}^{1/q} \] 
where \( d( \Delta) \) is the diameter of the convex body \( \Delta \). In particular, \( \phi \in \Pcal^{+}(\Rbb^m)\), hence has full \( g d \lambda \)-mass. 
\end{lem}


We will consider the following Sobolev space:
 \[ W^{1,q}(\Delta, gd\lambda) = \set{h \in L^q(\Delta, g d \lambda), \; \text{weak first derivatives of \(h\) are in \(L^q(\Delta,gd\lambda)\)}}\]

\begin{proof}
Let \( \phi \) be a weak solution with \( \sup_{\Delta}(\phi - \phi_{\Delta}) = 0 \). The Legendre transform \(f :=  \phi^{*} \) defines a convex function on \( \text{Int}(\Delta) \). Since \( \phi \in \Pcal^{+}(\Rbb^m) \), we have \( f \in L^q(\Delta,gd\lambda) \) because the Legendre transform is a bijection between \( \Pcal^{+}(\Rbb^m) \) and bounded convex functions on \( \Delta \) \cite[Proposition 2.3]{BB13}. Using \( \MA_g(\phi) = \mu = (d\phi^{*})_{*}(d \lambda) \) (cf. \cite[Lemma 2.7]{BB13}), we have 
\[ \int_{\Delta} \abs{d f(p)}^{\alpha} g(p) d\lambda(p) = \int_{\Rbb^m} \abs{x}^{\alpha} \MA_g(\phi) = \int_{\Rbb^m} \abs{x}^{\alpha} \mu , \quad \forall \alpha > 0 \] 
hence \( f \in W^{1,q}(\Delta, g d\lambda) \). 

If \( q > m \), then by Sobolev's inequality for the continuous embedding \( W^{1,q} (\Delta, g \lambda) \subset L^{\infty}(\Delta, g d \lambda) \):
\[ \sup_{\Delta} \abs{f} \leq \frac{1}{\vol (\Delta)} \int_{\Delta} \abs{f} gd\lambda  + C_{q,m} \frac{d(\Delta)^{(1-m/q)}}{\vol(\Delta)} \tuple{\int_{\Delta} \abs{d f}^q gd \lambda}^{1/m} \]
The proof is exactly the same as in e.g. \cite[Theorem IX.12]{Brez83}, where it is given when \( \Delta \) is a cube of length \(d\) and containing \( 0 \) in the interior, but this generalizes word-by-word to any convex body of diameter \( d \), containing \(0\) in its interior. The second integral term on the rhs is exactly 
\[ \tuple{\int_{\Rbb^m} \abs{x}^q \mu}^{1/q} \] 
Now let \(\phi_{\Delta}\) the support function of \( \Delta \). Using that \(-\sup(\phi-\phi_{\Delta}) = \inf_{\Delta} f \) \cite[Proposition 2.3]{BB13} and the normalization condition \( \sup( \phi - \phi_{\Delta}) = 0 \), we have:  
\[ -\sup(\phi - \phi_{\Delta}) = \inf_{\Delta} f = 0 \]
Let \(p_0 \in \Rbb^m \) such that \( f(p_0) = \inf_{\Delta} f \). By the convexity of \(f\), the differential exists Lebesgue-a.e. (hence \( g d \lambda \)- a.e.), so we have \( g d \lambda\)-a.e:
\[ \abs{f(p)} \leq d f(p). (p-p_0), \] 
which implies by Cauchy-Schwartz inequality:
\[ \int_{\Delta} \abs{f(p)} g(p) d \lambda(p) \leq d(\Delta) \int_{\Delta} \abs{df(p)} g(p) d \lambda(p) = d(\Delta) \int_{\Rbb^m} \abs{x} \mu \] 
This allows us to conclude.
\end{proof}

\section{Examples of Calabi-Yau horospherical cones} \label{Examples}
\subsection{Regular Calabi-Yau cones}
\subsubsection{An easy example}

\begin{ex}
Consider the conical embedding of \( SL_3/ U^{-} \) in Ex. \ref{rank_two_horospherical}. Fixing a root system \( (SL_3, B, T) \) with positive roots \( (\alpha_1, \alpha_2, \alpha_1 + \alpha_2) \). The linear function associated to the canonical divisor is \( \beta = \alpha_1 + \alpha_2 \). 

The cone admits a regular conical Calabi-Yau metric for the normalized (opposite) Reeb vector \( \xi_{-} = (\alpha_1^{\vee} + \alpha_2^{\vee})/4 \). Although it is possible to prove this by direct computations, we will employ a symmetry argument. Note that the volume minimization condition is invariant with respect to exchanging two roots, hence by uniqueness of the solution, the opposite Reeb vector is also invariant by this symmetry. Therefore, letting \( \xi_{-} = (x \alpha_1^{\vee} + \alpha_2^{\vee})/4 \), we have  \( x \alpha_1^{\vee} + \alpha_2^{\vee} = x \alpha_2^{\vee} + \alpha_1^{\vee} \), hence \( x = 1\). 

\begin{figure}[H]
\begin{tikzpicture}
\pgfmathsetmacro\ax{2}
\pgfmathsetmacro\ay{0}
\pgfmathsetmacro\bx{2 * cos(90)}
\pgfmathsetmacro\by{2 * sin(90)}
\pgfmathsetmacro\lax{2*\ax/3 + \bx/3}
\pgfmathsetmacro\lay{2*\ay/3 + \by/3}
\pgfmathsetmacro\lbx{\ax/3 + 2*\bx/3}
\pgfmathsetmacro\lby{\ay/3 + 2*\by/3}

\tikzstyle{couleur_pl}=[circle,draw=black!50,fill=blue!20,thick, inner sep = 0pt, minimum size = 2mm]

\fill [black!20] (0,0)--(\ax,\ay)--(\bx,\by) -- cycle;
\node at (\ax, \ay) [couleur_pl] {};
\node at (\bx, \by) [couleur_pl] {};
\draw[->, ultra thick] (0,0) -- (\ax,\ay) node[below right] {\( \rho(D_1) \)};
\draw[->, ultra thick] (0,0) -- (\bx, \by) node[above left] {\(\rho(D_2) \)};
\draw[->, ultra thick] (0,0) -- (0.25*\ax + 0.25*\bx, 0.25*\ay + 0.25*\by) node[above right] {\( \xi_{-} \)};

\draw (-2,0)--(2,0);
\draw (0,-2)--(0,2);



\end{tikzpicture}
\caption{}
\end{figure}

\end{ex}

\subsubsection{Examples arising from symmetric spaces}

Let us give a brief overview about the structure of (algebraic) symmetric spaces. The reader may consult \cite{Vus74}, \cite{Vus90} or \cite{dCP83} for a detailed treatment.  
Let \( G \) be a complex connected semisimple group, endowed with an algebraic involution $\sigma \neq \text{Id}$. Let $H$ be a closed subgroup such that $G^{\sigma} \subset H \subset N_G(G^{\sigma})$. The homogeneous space $G/H$ is called \textit{a symmetric space}. It is well-known that a symmetric space is spherical (see \cite{Vus74}).  

Let \( T_s \subset G \) be a maximal torus satisfying \( \sigma(t) = t^{-1}, \; \forall\; t \in T_s \). The rank of the symmetric space coincides with the dimension of \( T_s \). After choosing a Lévi subgroup adapted to \( H \), we have an identification  \( \Mcal(G/H) = \Mcal(T_s / T_s \cap H) \). Moreover, \( T_s / T_s \cap H \) is a subgroup of finite index in \( T_s \) (cf. \cite[2.2]{Vus90}), whence 
\[ \Mcal(G/H)_{\Rbb} = \Mcal(T_s)_{\Rbb} \] 
Let \( \tfrak_{s} \) be the Lie algebra of \( T_s \). The  non-compact part \( \tfrak_{nc,s} \) of \( \tfrak_s \) is isomorphic to \( \Ncal(T_s)_{\Rbb} \simeq \Hom(\Cbb^{*},T_s) \otimes \Rbb \).  

\begin{defn} \cite{dCP83} \cite{Vus76}
\begin{itemize}
\item There exists a \( \sigma \)-stable maximal torus \( T \) containing \( T_s \). 

\item Let \( \Phi \) the root system \( (G,T) \). The \emph{restricted root system of \( G/H \)}, which has the same rank as \( G/H \), is defined by:
\[ R := \set{ \ol{\alpha} \neq 0, \;  \ol{\alpha} = \alpha - \sigma(\alpha), \; \alpha \in \Phi} \]   
In particular, \( \ol{\alpha}|_{\tfrak_s} = 2 \alpha|_{\tfrak_s} \). 
\item After fixing a restricted root system, the valuation cone \( \Vcal \) of a symmetric space can be considered as the \textit{negative restricted Weyl chamber}, which is always strictly convex (if \( G/H \) is semisimple) and polyhedral of maximal dimension.
\end{itemize}
\end{defn}

The multiplicity of an element \( \ol{\alpha} \) is 
\[ m_{\ol{\alpha}} := \# \set{ \gamma \in \Phi, \; \ol{\alpha} = \gamma - \sigma(\gamma)} \] 
The set of multiplicities of a symmetric space is defined as the set of multiplicities of its restricted root system. Let \( n \) and \( r \) be respectively the dimension and the rank of \( G/H \). We have:
\[ n = r + \sum_{\ol{\alpha} \in R } m_{\ol{\alpha}} \] 

\begin{defn} \cite[Lemma 2.3]{Vus90}
Let \( \alpha \) be a root of \( \Phi \). The \emph{restricted coroot} of \(\ol{\alpha}\) is defined as 
\begin{itemize}
    \item \( (\ol{\alpha})^{\vee} = \frac{1}{2} \alpha^{\vee}\) if \(\sigma(\alpha) = - \alpha \) 
    \item \( (\ol{\alpha})^{\vee} = \frac{1}{2} (\alpha^{\vee} - \sigma(\alpha)^{\vee}), \) if  \( \sprod{\alpha^{\vee}, \sigma(\alpha)} = 0 \). 
    \item \( (\ol{\alpha})^{\vee} = \alpha^{\vee} - \sigma(\alpha)^{\vee}, \) if \( \sprod{\alpha^{\vee}, \sigma(\alpha)} = 1\). 
\end{itemize}
\end{defn}

\begin{prop} \cite[2.4, Proposition 1]{Vus90} \label{proposition_image_colors_symmetric_spaces}
The images of the colors in \( \Ncal(T_s/ T_s \cap H) \) are in bijection with the restricted coroots of  \( G/H \). More precisely, let \( \Dcal \) be the colors of \( G/H \). Then 
\[ \rho( \Dcal) = \set{ (\ol{\alpha})^{\vee}, \ol{\alpha} \in R } \] 
If \( G/H \) is not hermitian (i.e. the center of \(G^{\sigma}\) is discrete), then \( \rho \) is injective on \( \Dcal \); if not, then each fiber of \( \rho \) contains at most two points. 
\end{prop}

We denote by 
\[ \varpi_R = \sum_{\ol{\alpha} \in R^{+}} m_{\ol{\alpha}} \ol{\alpha} \] 
the sum of positive restricted roots (counted with multiplicities). The Killing form \( \kappa \) of \( \tfrak_s \) allows us to identify \( \Ncal_{\Rbb} \) with \( \Mcal_{\Rbb} \).

Now consider an irreducible semisimple symmetric space \( G/H \) of rank \( 2 \) with valuation cone \( \Vcal \). For all \( \xi  \in \text{Int} \Vcal_{\Qbb} \), the colored cones generated by \( \xi \) and the images of the colors cover the whole valuation cone \( \Vcal \), hence define a compactification \( X^{\xi} \) (cf. Thm. \ref{spherical_embeddings_classification}). The unique \( G \)-stable divisor \( D_0^{\xi} \) of \( X^{\xi} \) corresponds to the ray generated by \( \xi \). 

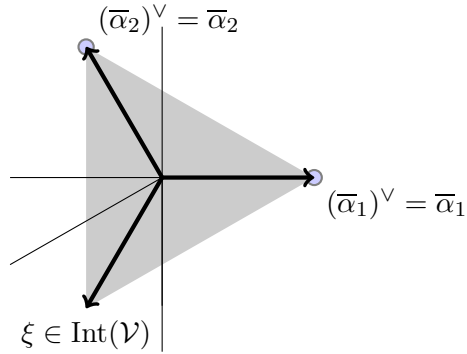
\begin{figure}[H]
\begin{tikzpicture}
\pgfmathsetmacro\ax{2}
\pgfmathsetmacro\ay{0}
\pgfmathsetmacro\bx{2 * cos(120)}
\pgfmathsetmacro\by{2 * sin(120)}
\pgfmathsetmacro\lax{2*\ax/3 + \bx/3}
\pgfmathsetmacro\lay{2*\ay/3 + \by/3}
\pgfmathsetmacro\lbx{\ax/3 + 2*\bx/3}
\pgfmathsetmacro\lby{\ay/3 + 2*\by/3}

\tikzstyle{couleur_pl}=[circle,draw=black!50,fill=blue!20,thick, inner sep = 0pt, minimum size = 2mm]

\fill [black!20] (0,0)--(\ax,\ay)--(\bx,\by) -- cycle;
\node at (\ax, \ay) [couleur_pl] {};
\node at (\bx, \by) [couleur_pl] {};
\fill [black!20] (0,0) --(-\lax - \lbx, -\lay - \lby)--(\ax, \ay) -- cycle;
\fill [black!20] (0,0) --(-\lax - \lbx, -\lay - \lby)--(\bx, \by) -- cycle;
\draw[->, ultra thick] (0,0) -- (\ax,\ay) node[below right] {\( (\ol{\alpha}_1)^{\vee} = \ol{\alpha}_1 \)};
\draw[->, ultra thick] (0,0) -- (\bx, \by) node[above right] {\( (\ol{\alpha}_2)^{\vee} = \ol{\alpha}_2 \)};
\draw[thin] (-2,0)--(2,0);
\draw[thin] (0,-1.7)--(0,2);

\draw (0,0) -- (-2*\lbx, -2*\lby);
\draw (0,0) -- (-2*\lax, -2*\lay);
\draw[->, ultra thick] (0,0) -- (-\lax - \lbx, -\lay - \lby) node[below] {\( \xi \in \text{Int}(\Vcal) \)};


\end{tikzpicture}
\caption{Compactification of \( A_2 \)-symmetric spaces with a unique \( G\)-stable divisor.}
\end{figure}

The open \( G \)-orbit of \( D_0^{\xi} \) is then \( G\)-isomorphic to a horospherical space \( G /H_0 \). This can be seen from the combinatorial data \( (\Mcal_0, \Ncal_0, \Vcal_0, \Dcal_0) \) of \( G / H_0 \), which in turn can be read from \( G/H \) thanks to the work of Brion-Pauer  \cite{BP87}. 

Let us briefly describle the combinatorial data of \( G/H_0 \). Let \( \Mcal_0 \) be the weight lattice of \( G / H_0 \). Then 
\[ \Mcal_0 = (\Mcal \oplus \Zbb )  \cap (\Nbb \xi)^{\perp} \] 
Let \( \Ncal_0 \) be the lattice of coweights. Define:
\[ \pi_0 : \Ncal \oplus \Zbb \to \Ncal_0 \] 
as the natural dual application to the inclusion \( \Mcal_0 \subset \Mcal \). We have \( \Vcal_0 = \pi_0(\Vcal ) = (\Ncal_0)_{\Qbb} \) (cf. \cite[Théorème 3.6]{BP87}), so \( G/H_0 \) is horospherical. By \textit{loc. cit.}, the (left)-stabilizer \( Q_0 \) of the open Borel-orbit of \( G/H_0 \) is exactly the stabilizer of the open Borel-orbit of \( G/H \). Let \( P_0 \) be the opposite parabolic subgroup to \( Q_0 \). Let \( I_0 \) be the positive simple roots in \( G \) associated to \( P_0 \). We then have by Prop. \ref{horospherical_colors_description_prop}:
\[ \rho_0(\Dcal_0) = \set{ \alpha^{\vee}|_{\Mcal_0}, \quad \alpha \in S \backslash I_0} \] 
where \( S \) is the set of positive simple roots of \( G \), and one can show that \( \rho_0(\Dcal_0) \) actually consists of multiples of \( \pi_0((\ol{\alpha}_1)^{\vee}) \) and \( \pi_0((\ol{\alpha}_2)^{\vee}) \) (cf. \cite[2.3, Remarques 2)]{Vus90}). 

Consider now the horospherical space \( G/H_0 \times \Cbb^{*} \). It is immediate that \( \Mcal(G/H_0 \times \Cbb^{*}) = \Mcal_0 \oplus \Zbb \). Its dual lattice  \( \Ncal(G/H_0 \times \Cbb^{*}) \) is then \( \Ncal_0 \oplus \Zbb \). In particular \( \Ncal(G/H_0 \times \Cbb^{*})_{\Rbb} \simeq \Ncal_{\Rbb} \). The colors of \( G/H_0 \times \Cbb^{*} \) are exactly \( \Dcal_0 \times \Cbb^{*} := \set{D \times \Cbb^{*}, D \in \Dcal_0} \). 

Let \( \Ccal_0 \) be the colored cone generated by the images of \( \Dcal_0 \times \Cbb^{*} \); then \( \Ccal_0 \) coincides with the cone generated by restricted coroots. It follows that \( \Ccal_0 \) does not depend on the choice of \( \xi \).  Denote by \( \Ccal_0^{\vee} \) the dual cone, generated by primitive directions orthogonal to \( \rho_0(\Dcal_0 \times \Cbb^{*})  \); then \( \Ccal^{\vee}_0 \) is exactly the negative Weyl chamber \( \Vcal \). 

\begin{figure}[H]
\begin{tikzpicture}
\pgfmathsetmacro\ax{2}
\pgfmathsetmacro\ay{0}
\pgfmathsetmacro\bx{2 * cos(120)}
\pgfmathsetmacro\by{2 * sin(120)}
\pgfmathsetmacro\lax{2*\ax/3 + \bx/3}
\pgfmathsetmacro\lay{2*\ay/3 + \by/3}
\pgfmathsetmacro\lbx{\ax/3 + 2*\bx/3}
\pgfmathsetmacro\lby{\ay/3 + 2*\by/3}

\tikzstyle{couleur_pl}=[circle,draw=black!50,fill=blue!20,thick, inner sep = 0pt, minimum size = 2mm]

\fill [black!20] (0,0)--(\ax,\ay)--(\bx,\by) -- cycle;
\draw[->, ultra thick] (0,0) -- (\ax,\ay) node[below right] {\( (\ol{\alpha}_1)^{\vee} = \ol{\alpha}_1 \)};
\draw[->, ultra thick] (0,0) -- (\bx, \by) node[above right] {\( (\ol{\alpha}_2)^{\vee} = \ol{\alpha}_2 \)};

\draw (-2,0)--(2,0);
\draw (0,-2)--(0,2);
\node at ( 1.7*\lax - 1.7*\lbx, 1.7*\lay - 1.7*\lby) [couleur_pl] {};
\draw[->] (0,0) -- ( 1.7*\lax - 1.7*\lbx, 1.7*\lay - 1.7*\lby) node[below right] { \( \pi_0((\ol{\alpha}_1)^{\vee}) \) };
\draw (2*\lax - 2*\lbx, 2*\lay - 2*\lby)--(2*\lbx - 2*\lax, 2*\lby - 2*\lay);

\draw (0,0) -- (-2*\lbx, -2*\lby);
\draw (0,0) -- (-2*\lax, -2*\lay);
\fill [blue!20] (0,0) -- (-2*\lax, -2*\lay) -- (-2*\lbx, -2*\lby) -- cycle;
\node[] at (-0.5,-1) {\( \Ccal_0^{\vee}  \) }; 
\node[] at (0.5,0.5) {\( -\Ccal_0\)};

\end{tikzpicture}

\caption{Horospherical cone constructed from a \( A_2 \)-symmetric space. Note that \( \rho_0(\Dcal_0) \) consists exactly of the images of the restricted coroots by \( \pi_0 \).}
\end{figure}
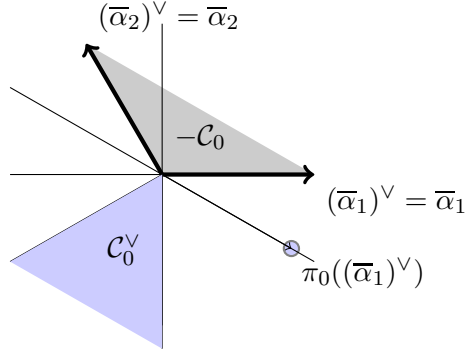

\begin{lem}
The embedding \( G/H_0 \times \Cbb^{*} \hookrightarrow Y_0 \) associated to \( (\Ccal_0, \Dcal_0 \times \Cbb^{*} ) \) is a horospherical cone.
\end{lem}

\begin{proof}
The cone  \( \Ccal_0 \)  is a polyhedral cone of maximal dimension containing all the colors, hence defines a horospherical cone by Thm. \ref{spherical_cone_criterion}.
\end{proof}

Now consider the real vector space \( \tfrak_{nc,s}^{*} \) with basis \( (\delta, \tau ) \) where \( \tau \) is in the interior of the positive Weyl chamber, such that: 
\[ \sprod{\tau ,\xi} = -1, \quad \kappa(\tau, \delta ) = 0, \quad \kappa(\ol{\alpha}_1, \delta) < 0 \] 
We are interested in finding a vector \( \tau \) that defines the conical Calabi-Yau metric on \( Y_0 \). The vector \( \tau \) might be interpreted as the opposite Reeb vector \( -\xi \) as we explained in the previous section.

For practical reasons, we parametrize the orthogonal direction \( \delta \) by:
\[ \delta = -x \wt{\alpha}_1 + \wt{\alpha}_2, \; x \in \Rbb \] 
Here \( \wt{\alpha}_1, \wt{\alpha}_2 \) denote the directions of the Weyl wall orthogonal to \( \ol{\alpha}_2, \ol{\alpha}_1 \) defined as
\[ \wt{\alpha}_1 = \ol{\alpha}_1 - \frac{\kappa(\ol{\alpha}_1,\ol{\alpha}_2)}{\kappa(\ol{\alpha}_2,\ol{\alpha}_2)} \ol{\alpha}_2, \quad \wt{\alpha}_2 = \ol{\alpha}_2 - \frac{\kappa(\ol{\alpha}_1, \ol{\alpha}_2)}{\kappa(\ol{\alpha}_1,\ol{\alpha}_1)} \ol{\alpha}_1 \] 
It is then easy to see that \( \tau \) belongs to the interior of the cone \( \Ccal_0 \) iff \( x > 0 \). 

Let
\[ P(p) = \prod_{\ol{\alpha} \in R^{+}} \kappa(\ol{\alpha}, p \delta  + \varpi_R)^{m_{\ol{\alpha}}}, \; p \in \tfrak_{nc,s}^{*}  \]
Consider the polytope \( \Delta_{\tau}  = [ \lambda_{-}(x), \lambda_{+}(x) ] \) where \( \lambda_{-} \) and \( \lambda_{+} \) are the intersections with the Weyl walls of the line containing \( \varpi_R \) and orthogonal to \( \tau \). In more explicit terms: 
\[ \lambda_{+}(x) = \frac{-\kappa(\ol{\alpha}_1, \varpi_R)}{ \kappa(\ol{\alpha}_1, \delta)} , \quad \lambda_{-}(x) = \frac{-\kappa(\ol{\alpha}_2, \varpi_R)}{\kappa(\ol{\alpha}_2, \delta)} \] 
By our main theorem, the polarized cone \( (Y_0,\xi) \) admits a conical Calabi-Yau metric if and only if: 
\[ \bar_P (\Delta_{\tau})= \varpi_R \] 
where \( \bar_P \) is the barycenter with respect to the measure \( P(p)dp \). Equivalently, we have: 

\begin{prop}
There exists a conical Calabi-Yau metric on \( Y_0 \) if and only if the polynomial equation equivalent to
\begin{equation} \label{polynomial_CY}
\int_{\lambda_{-}(x)}^{\lambda_{+}(x)} p P(p) dp = 0 
\end{equation} 
admits a (necessarily unique) positive root. In particular, \( Y_0 \) is a quasi-regular Calabi-Yau cone if and only if the root is rational. 
\end{prop}

\begin{ex} Consider the symmetric spaces \( G/H \) of rank \( 2 \) whose root system is of type \( A_2 \). The multiplicities of the two simple restricted roots \( \ol{\alpha}_1, \ol{\alpha}_2 \) all equal to \( m \), hence 
\[ \varpi_R = 2 m( \ol{\alpha}_1 + \ol{\alpha}_2 ) \]  
We also have:
\[ \kappa(\ol{\alpha}_1, \ol{\alpha}_1) = \kappa(\ol{\alpha}_2, \ol{\alpha}_2) = 1, \quad \kappa(\ol{\alpha}_1, \ol{\alpha}_2) = -1/2, \quad \kappa(\ol{\alpha}_1, \wt{\alpha}_1) = \kappa(\ol{\alpha}_2, \wt{\alpha}_2) = 3/4 \] 
so \( \kappa(\ol{\alpha}_1, \varpi_R) = \kappa(\ol{\alpha}_2, \varpi_R) = m \) and \( \kappa(\ol{\alpha}_1 + \ol{\alpha}_2, \varpi_R) = 2m \). It follows that the polynomial equation (\ref{polynomial_CY}) remains unchanged under the symmetry of \( \tfrak_{nc,s}^{*} \) exchanging two roots, hence the (opposite) Reeb vector is fixed (by uniqueness of the volume minimization condition) and \( \delta \) is sent to its opposite under this symmetry. It follows that
\[ \delta = -x \wt{\alpha}_1 + \wt{\alpha}_2 =  x \wt{\alpha}_2 - \wt{\alpha}_1 \] 
hence \( x = 1 \) and \( \tau = \ol{\alpha}_1 + \ol{\alpha}_2 \). As a consequence, the horospherical cones constructed from \[ SL_3/ SO_3, \; SL_3 \times SL_3, \;  SL_6/ Sp_6, \; E_6/ F_4 \]
(which are symmetric spaces of type \( A_2 \) with multiplicities \( m = 1,2,4,8 \) respectively) all have quasi-regular Calabi-Yau structures with a rational Reeb vector.  
\end{ex}

\subsection{Irregular Calabi-Yau cones}

We consider in this part rank-two symmetric spaces  with root system of type $BC_2$ of multiplicities \( (m_1,m_2,m_3) \). Let \( (\ol{\alpha}_1,\ol{\alpha}_2) \) be the simple restricted roots with multiplicities \( (m_1, m_2) \). The positive roots of multiplicity \( m_1 \) are \( (\ol{\alpha}_1, \ol{\alpha}_1 + 2 \ol{\alpha}_2) \), the ones with multiplicity \( m_2 \) are \( (\ol{\alpha}_2, \ol{\alpha}_1 + \ol{\alpha}_2) \), and those with multiplicity \( m_3 \) are \( (2 \ol{\alpha}_2, 2 \ol{\alpha}_1 + 2 \ol{\alpha}_2) \). 

We also have 
\begin{align*}
\kappa(\ol{\alpha}_1, \ol{\alpha}_1) &= 2, \; \kappa(\ol{\alpha}_1, \ol{\alpha}_2) = -1, \; \kappa(\ol{\alpha}_2, \ol{\alpha}_2) = 1, \\ &\kappa(\ol{\alpha}_1, \wt{\alpha}_1) = 1, \; \kappa(\ol{\alpha}_2, \wt{\alpha}_2) = 1/2 
\end{align*}
Moreover, 
\begin{align*} 
\varpi_R = (2m_1 + 2m_2 + 4 m_3) \ol{\alpha}_2 + (2m_1 + m_2 + 2m_3) \ol{\alpha}_1 
\end{align*} 
A direct calculation gives us (up to multiplying by a constant) :
\begin{align*}
P(p) = &( 2m_1 - p x)^{m_1} ( 2m_1 + 2m_2 + 4m_3 + p(1-x))^{m_1} \\
&(2m_2 + 4 m_3 + p)^{m_2 + m_3} (4m_1 + 2m_2 + 4m_3 + (1-2x)p )^{m_2 + m_3} 	
\end{align*} 
The vertices of the polytope \( \Delta_{\tau} \) have coordinates
\[ \lambda_{+}(x) = \frac{-\kappa(\ol{\alpha}_1, \varpi_R)}{ \kappa(\ol{\alpha}_1, \delta)} = \frac{2m_1}{x}, \quad \lambda_{-}(x) = \frac{-\kappa(\ol{\alpha}_2, \varpi_R)}{\kappa(\ol{\alpha}_2, \delta)} = - (2m_2 + 4 m_3) \]

\begin{ex}
Let \( \Gcal_k := SL_k / SL_2 \times SL_{k-2} \) be the complex grassmannian of multiplicities \( (m_1, m_2, m_3) = (2, 2k - 8,1), k \geq 5 \), which we can view as the tangent or cotangent bundle over a real grassmannian, endowed with an appropriate complex structure. If \( SL_k \) acts on \( \Cbb^k \) on the left, then geometrically, \( \Gcal_k \) is \( SL_k \)-isomorphic to a pair \( (V,W) \) of vector subspaces of \( \Cbb^{k} \) with dimensions \( (2, k-2) \) such that \( \Cbb^k = V \oplus W \). 
 
Consider the complex grassmannian \( \Gcal_{k=5} \), with multiplicities
\( (m_1,m_2,m_3) = (2,2,1) \). We have  :
\[ \varpi_R = 8 \ol{\alpha}_1 + 12 \ol{\alpha}_2, \quad \lambda_{-} = -6,\; \lambda_{+} = \frac{4}{x} \] 
The polynomial \( P \) then becomes
\[ P(p) = ( 4 - p x)^2 ( 12 + p(1-x) )^2 (16 + p(1-2x))^3 (8 - p)^3 \] 
and the barycenter condition reduces to
\[ (1 + 2 x)^{11} (-42 - 84 x - 63 x^2 + 60 x^3 + 240 x^4 + 336 x^5 + 
   224 x^6)  = 0 \] 
This equation has a unique irrational positive root, as follows from the Eisenstein's criterion for the prime number \( 3 \). In particular, there exists no compactification with a unique \( SL_5 \)-stable Kähler-Einstein divisor, but the horospherical cone \( Y_0 \) still has an irregular conical Calabi-Yau structure. 
\end{ex}

We believe that the following holds.

\begin{conjecture}
For each \( k \geq 5 \), the horospherical cone constructed from \( \Gcal_k \) has an irregular conical Calabi-Yau sructure. 
\end{conjecture}

\begin{ex}
Now consider the following symmetric spaces 
\[SO_5 \times SO_5/ SO_5, \; Sp_8 / Sp_4 \times Sp_4, \; SO_{10}/ GL_5, \; E_6/SO_{10} \times SO_2, \] 
They are respectively of multiplicities
\[(2,2,0), \; (3,4,0), \; (4,4,1), \; (6,8,1), \] 
The polynomial equations resulting from (\ref{polynomial_CY}) are, respectively: 
\begin{align*}
(1 + x)^9 (-7 - 15 x - 14 x^2 + 2 x^3 + 24 x^4 + 28 x^5) = 0 
\end{align*}
\begin{align*}
(3 + 4 x)^{15} (&-99 - 319 x - 545 x^2 - 549 x^3 - 140 x^4 \\
 &+ 616 x^5 + 1296 x^6 + 1360 x^7 + 704 x^8) = 0 
\end{align*}
\begin{align*}
(2 + 3 x)^{19} (&-6006 - 25025 x - 55770 x^2 - 80850 x^3 \\
&- 69300 x^4 + 4914 x^5 + 134820 x^6 \\ 
&+ 264880 x^7 + 314160 x^8 + 240240 x^9 + 
   96096 x^{10}) = 0
\end{align*}
\begin{align*}
(3 + 5 x)^{31} (&-6128925 - 38326211 x - 129851491 x^2 - 309121659 x^3\\
&-563633385 x^4 - 802569405 x^5 - 849852729 x^6 \\
&- 498060849 x^7 + 375429054 x^8 + 1679517840 x^9\\
&+ 3059056000 x^{10} + 4002942944 x^{11} + 4101349824 x^{12} \\
&+ 3312646656 x^{13} + 2032884480 x^{14} + 859541760 x^{15} \\
&+ 191009280 x^{16}) = 0
\end{align*}
One can check directly using e.g. \textit{Mathematica} that the unique positive solution of each polynomial is irrational. In particular, the horospherical cone constructed in each case has an irregular conical Calabi-Yau structure. As a consequence, we obtain examples of irregular horospherical cones. 
\end{ex}

\bibliographystyle{alpha}
\bibliography{biblio}
\end{document}